\documentclass[11pt]{article}
\usepackage{graphicx,psfrag}
\usepackage{amssymb,amsfonts,amsthm}
\usepackage{amsmath,amsthm,amssymb}
\usepackage{amsfonts}
\usepackage{relsize}
\usepackage[T1]{fontenc}
\usepackage[utf8]{inputenc}
\usepackage{graphicx}
\usepackage{color}
\usepackage{subcaption}
\usepackage[font=scriptsize]{caption}

\newcommand{\DG}{\mathrm{DG}}
\newcommand{\Stab}{\mathrm{Stab}}

\newcommand{\Top}{\mathbf{top}(\Delta)}

\newcommand{\la}{\langle}
\newcommand{\ra}{\rangle}

\renewcommand{\P}{\mathcal{P}}

\newtheorem{theorem}{Theorem}[section]
\newtheorem{lemma}[theorem]{Lemma}

\newtheorem{cor}[theorem]{Corollary}
\newtheorem{proposition}[theorem]{Proposition}
\theoremstyle{definition}
\newtheorem{definition}[theorem]{Definition}
\newtheorem{remark}[theorem]{Remark}

\newcommand{\call}{{\mathcal L}}

\newcommand {\N}{\mathbb{N}} 
\newcommand {\R}{\mathbb{R}} 

\newcommand {\iv}{^{-1}}

\newcommand{\D}{\Delta}
\newcommand{\ini}{\iota(\Delta)}
\newcommand{\ter}{\tau(\Delta)}

\numberwithin{equation}{section}

\newcounter{AbcT}

\newcommand{\La}{\mathcal{L}}
\newcommand{\G}{\mathcal{G}}

\newcommand{\nc}{\newcommand}
\nc{\meet}{\wedge}
\nc{\op}{\operatorname}\nc{\FP}{\op{FP}}\nc{\FS}{\op{FS}}\nc{\FPhat}{\widehat{\op{FP}}}

\newtheorem {Theoremintro}    {Theorem}
\newtheorem {Corollaryintro}[Theoremintro]    {Corollary}

\newtheorem {Theorem}    {Theorem}[section]
\newtheorem* {Theorem2}    {Theorem 2}

\newtheorem {Problem}    [Theorem]{Problem}

\newtheorem {Lemma}      [Theorem]    {Lemma}

\newtheorem {Example}      [Theorem]    {Example}

\theoremstyle{remark}
\newtheorem {Remark}		 [Theorem]    {\bf{Remark}}

\newcommand{\topp}{{\bf top}}

\newcommand{\bott}{{\bf bot}}
\newcommand{\1}{\mathbf{1}}
\newcommand{\arr}{\overrightarrow}

\usepackage{lipsum}

\begin{document}

\title{On Jones' subgroup of R. Thompson group $F$}

%

\author{Gili Golan, 
 Mark Sapir\thanks{The
research was supported in part by the NSF grants DMS 1418506,
DMS	1318716 and by the BSF grant 2010295.}}
\maketitle
\abstract{Recently Vaughan Jones showed that the R. Thompson group $F$ encodes in a natural way all knots and links in $\R^3$, and a certain subgroup $\overrightarrow F$ of $F$ encodes all oriented knots and links. We answer several questions of Jones about $\overrightarrow F$. In particular we prove that the subgroup $\overrightarrow F$ is generated by $x_0x_1, x_1x_2, x_2x_3$ (where $x_i,i\in \N$ are the standard generators of $F$) and is isomorphic to $F_3$, the analog of $F$ where all slopes are powers of $3$ and break points are $3$-adic rationals. We also show that $\overrightarrow F$ coincides with its commensurator. Hence the linearization of the permutational representation
of $F$ on $F/\overrightarrow F$ is irreducible. We show how to replace $3$ in the above results by an arbitrary $n$, and to construct a series of irreducible representations of $F$ defined in a similar way. Finally we analyze Jones' construction and deduce that the Thompson index of a link is linearly bounded in terms of the number of crossings in a link diagram.}

\tableofcontents

\section{Introduction}

A recent result of Vaughan Jones \cite{Jo} shows that Thompson group $F$ encodes in a natural way all links (this construction is presented in Section \ref{s:6} below). A subgroup of $F$, called by Jones the \emph{directed Thompson group} $\overrightarrow F$, encodes all oriented links.
In order to define $\overrightarrow F$, Jones associated with every element $g$ of $F$ a graph $T(g)$ using the description of elements of $F$ as pairs of binary trees  (see Section \ref{sec:graph} for details). The group $\overrightarrow F$ is the set of all elements in $F$ for which the associated graph $T(g)$ is bipartite. Jones asked for an abstract description of the subgroup $\overrightarrow F$. For example, it is not clear from the definition whether or not $\overrightarrow F$ is finitely generated.

We define the graph $T(g)$ in a different (but equivalent) way. By \cite{GS} $F$ is a diagram group. For every diagram $\Delta$ in $F$ the graph $T(\Delta)$ is a certain subgraph of $\Delta$. 
 Then, the subgroup of $F$ composed of all reduced diagrams $\Delta$ in $F$ with $T(\Delta)$ bipartite is Jones' subgroup $\overrightarrow F$. Using this definition we
 give several descriptions of $\overrightarrow F$. Recall that for every $n\ge 2$ one can define a ``brother'' $F_n$ of $F=F_2$ as the group of all piecewise linear increasing homeomorphisms of the unit interval where all slopes are powers of $n$ and all breaks of the derivative occur at $n$-adic fractions, i.e., points of the form $\frac a{n^k}$ where $a, k$ are positive integers \cite{B}. It is well known that $F_n$ is finitely presented for every $n$ (a concrete and easy presentation can be found in \cite{GS}).


\begin{Theoremintro}\label{thm:1}
Jones' subgroup $\overrightarrow F$ is generated by elements $x_0x_1,$ $x_1x_2, x_2x_3$ where $x_i,i\in \N,$ are the standard generators of $F$. It is isomorphic to $F_3$ and coincides with the smallest subgroup of $F$ which contains $x_0x_1$ and is closed under addition (which is a natural binary operation on $F$, see Section \ref{sec:fdg}).
\end{Theoremintro}

This theorem implies the following characterization of $\overrightarrow F$ which can be found in \cite{Jo}.

\begin{Theoremintro}\label{thm:2}
Jones' subgroup $\overrightarrow F$ is the stabilizer of the set of dyadic fractions from the unit interval $[0,1]$ with odd sums of digits, under the standard action of $F$ on the interval $[0,1]$.
\end{Theoremintro}

As a corollary from Theorem \ref{thm:2} we get the following statement answering a question by Vaughan Jones.

\begin{Corollaryintro}\label{thm:3}
Jones' subgroup $\overrightarrow F$ coincides with its commensurator in $F$.
\end{Corollaryintro}

As noted in \cite{Jo}, this implies that the linearization of the permutational representation of $F$ on $F/\overrightarrow F$ is irreducible (see \cite{Mac}).

In \cite{Jo}, Jones introduced the Thompson index of a link which can be defined as the smallest number of leaves  of a tree diagram (= the number of vertices minus one in the semigroup diagram) representing that link. The construction in \cite{Jo} does not give an estimate of the Thompson index. But analyzing and slightly modifying that construction (using some results from Theoretical Computer Science) we prove that the Thompson index of a link does not exceed 12 times the number of crossings in any link diagram.

The paper is organized as follows. In Section \ref{sec:pre} we give some preliminaries on Thompson group $F$. In Section \ref{sec:graph} we  give Jones' definition of the Thompson graph associated with an element of $F$, we also give the definition in terms of semigroup diagrams \cite{GS} and prove the equivalence of two definitions. In Section \ref{sec:properties} we define Jones subgroup $\overrightarrow F$ and prove Theorems \ref{thm:1} and \ref{thm:2} and Corollary \ref{thm:3}. Section \ref{sec:5} contains generalizations of the previous results for arbitrary $n$. In particular, we show that for every $n\ge 2$, the smallest subgroup of $F$ containing the element $x_0\cdots x_{n-2}$ and closed under addition, is isomorphic to $F_n$, can be characterized in terms of the graphs $T(\Delta)$,  is the intersection of stabilizers of certain sets of binary fractions, and coincides with its commensurator in $F$. Although Theorems \ref{thm:1} and \ref{thm:2} are special cases of the results of Section \ref{sec:5}, the direct proofs of these results are much less technical while ideologically similar, and the original questions of Jones concerned the case $n=3$ only. Thus we decided to keep the proofs of Theorems \ref{thm:1} and \ref{thm:2}. In Section \ref{sec:5} we show how to adapt these proofs to the general case. In Section \ref{s:6}, we analyze Jones' construction and prove the linear upper bound for the Thompson index of a link.
\vskip .5cm

{\bfseries Acknowledgments.} We are grateful to Vaughan Jones for asking questions about $\overrightarrow F$. We are also grateful to Victor Guba and Michael Kaufmann for helpful conversations.

\section{Preliminaries on $F$}\label{sec:pre}

\subsection{$F$ as a group of homeomorphisms}

The most well known definition of the R. Thompson group $F$ is this (see \cite{CFP}): $F$ consists of all piecewise-linear increasing self-homeomorphisms of the unit interval with all slopes powers of $2$ and all break points of the derivative dyadic fractions. The group $F$ is generated by two functions $x_0$ and $x_1$ defined as follows.

	
	\[
   x_0(t) =
  \begin{cases}
   2t &  :  0\le t\le \frac{1}{4} \\
   t+\frac14       & : \frac14\le t\le \frac12 \\
   \frac{t}{2}+\frac12       & : \frac12\le t\le 1
  \end{cases} 	\qquad	
   x_1(t) =
  \begin{cases}
   t &  : 0\le t\le \frac12 \\
   2t-\frac12       & : \frac12\le t\le \frac{5}{8} \\
   t+\frac18       & : \frac{5}{8}\le t\le \frac34 \\
   \frac{t}{2}+\frac12       & : \frac34\le t\le 1 	
  \end{cases}
\]


One can see that $x_1$ is the identity on $[0,\frac12]$ and a shrank by the factor of 2 copy of $x_0$ on $[\frac12, 1]$. The composition in $F$ is from left to right.

Equivalently, the group $F$ can be defined using dyadic subdivisions \cite{B}. We call a subdivision of $[0, 1]$ a \emph{dyadic subdivision} if it is obtained by repeatedly cutting intervals in half. If $S_1,S_2$ are dyadic subdivisions with the same number of pieces, we can define a piecewise linear map taking each segment of the subdivision $S_1$ linearly to the corresponding segment of $S_2$. We call such a map a \emph{dyadic rearrangement}. The group $F$ consists of all dyadic rearrangements.

\subsection{$F$ as a diagram group}\label{sec:fdg}

It was shown in \cite[Example 6.4]{GS} that the R. Thompson group $F$ is a diagram group over the semigroup
presentation $\la x\mid x=x^2\ra$.

Let us recall the definition of a {\em diagram group} (see \cite{GS,GS1} for more formal definitions). A (semigroup)
{\em diagram} is a plane directed labeled graph tesselated by cells, defined up to an isotopy of the plane. Each
diagram $\Delta$ has the top path $\topp(\Delta)$, the bottom path $\bott(\Delta)$, the initial and terminal vertices
$\iota(\Delta)$ and $\tau(\Delta)$. These are common vertices of $\topp(\Delta)$ and $\bott(\Delta)$.  The whole
diagram is situated between the top and the bottom paths, and every edge of $\Delta$ belongs to a (directed) path in
$\Delta$ between $\iota(\Delta)$ and $\tau(\Delta)$. More formally, let $X$ be an alphabet. For every $x\in X$ we
define the {\em trivial diagram} $\varepsilon(x)$ which is just an edge labeled by $x$. The top and bottom paths of
$\varepsilon(x)$ are equal to $\varepsilon(x)$, the vertices $\iota(\varepsilon(x))$ and $\tau(\varepsilon(x))$ are the initial and
terminal vertices of the edge. If $u$ and $v$ are words in $X$, a {\em cell} $(u\to v)$ is a plane graph consisting of
two directed labeled paths, the top path labeled by $u$ and the bottom path labeled by $v$, connecting the same points
$\iota(u\to v)$ and $\tau(u\to v)$. 
 There are three operations that can be applied to diagrams in order to obtain new
diagrams.

1. {\bf Addition.} Given two diagrams $\Delta_1$ and $\Delta_2$, one can identify $\tau(\Delta_1)$ with
$\iota(\Delta_2)$. The resulting plane graph is again a diagram denoted by $\Delta_1+\Delta_2$, whose top (bottom)
path is the concatenation of the top (bottom) paths of $\Delta_1$ and $\Delta_2$. If $u=x_1x_2\ldots x_n$ is a word in
$X$, then we denote $\varepsilon(x_1)+\varepsilon(x_2)+\cdots + \varepsilon(x_n)$ ( i.e. a simple path labeled by $u$)
by $\varepsilon(u)$  and call this diagram also {\em trivial}.

2. {\bf Multiplication.} If the label of the bottom path of $\Delta_1$ coincides with the label of the top path of
$\Delta_2$, then we can {\em multiply} $\Delta_1$ and $\Delta_2$, identifying $\bott(\Delta_1)$ with $\topp(\Delta_2)$.
The new diagram is denoted by $\Delta_1\circ \Delta_2$. The vertices $\iota(\Delta_1\circ \Delta_2)$ and
$\tau(\Delta_1\circ\Delta_2)$ coincide with the corresponding vertices of $\Delta_1, \Delta_2$, $\topp(\Delta_1\circ
\Delta_2)=\topp(\Delta_1),
\bott(\Delta_1\circ \Delta_2)=\bott(\Delta_2)$.

3. {\bf Inversion.} Given a diagram $\Delta$, we can flip it about a horizontal line obtaining a new diagram
$\Delta\iv$ whose top (bottom) path coincides with the bottom (top) path of $\Delta$.

\begin{figure}[h!]
\begin{center} 
\unitlength=1mm
\special{em:linewidth 0.4pt}
\linethickness{0.4pt}
\begin{picture}(124.41,55.00)
\put(1.00,30.00){\circle*{2.00}}
\put(46.00,30.00){\circle*{2.00}}
\put(1.00,30.00){\line(1,0){45.00}}
\bezier{320}(1.00,30.00)(24.00,55.00)(46.00,30.00)
\bezier{332}(1.00,30.00)(24.00,5.00)(46.00,30.00)
\put(24.00,35.00){\makebox(0,0)[cc]{$\Delta_1$}}
\put(24.00,25.00){\makebox(0,0)[cc]{$\Delta_2$}}
\put(24.00,10.00){\makebox(0,0)[cc]{$\Delta_1\circ\Delta_2$}}
\put(66.00,30.00){\circle*{2.00}}
\put(94.00,30.00){\circle*{2.00}}
\put(123.00,30.00){\circle*{2.00}}
\bezier{164}(66.00,30.00)(80.00,45.00)(94.00,30.00)
\bezier{152}(66.00,30.00)(81.00,17.00)(94.00,30.00)
\bezier{172}(94.00,30.00)(109.00,46.00)(123.00,30.00)
\bezier{168}(94.00,30.00)(110.00,15.00)(123.00,30.00)
\put(80.00,30.00){\makebox(0,0)[cc]{$\Delta_1$}}
\put(109.00,30.00){\makebox(0,0)[cc]{$\Delta_2$}}
\put(94.00,10.00){\makebox(0,0)[cc]{$\Delta_1+\Delta_2$}}
\end{picture}
\end{center}
\caption{The multiplication and addition of diagrams.}
\label{f1}
\end{figure}
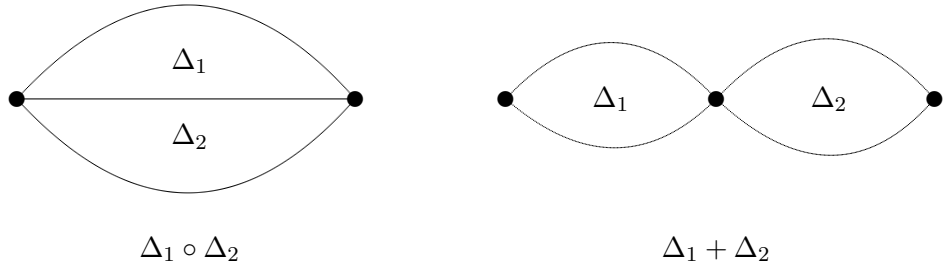

\begin{definition} A diagram over a collection of cells (i.e., a semigroup presentation) $\P$ is any plane
graph obtained from the trivial diagrams and cells of $\P$ by the operations of addition, multiplication and inversion.
If the top path of a diagram $\Delta$ is labeled by a word $u$ and the bottom path is labeled by a word $v$, then we
call $\Delta$ a $(u,v)$-diagram over $\P$.
\end{definition}

Two cells in a diagram form a {\em dipole} if the bottom part of the first cell coincides with the top part of the
second cell, and the cells are inverses of each other. Thus a dipole is a subdiagram of the form $\pi\circ \pi\iv$ where $\pi$ is a cell.
In this case, we can obtain a new diagram by removing the two cells
and replacing them by the top path of the first cell. This operation is called {
\em elimination of dipoles}. The new diagram is called {\em equivalent}
to the initial one. A diagram is called {\em reduced} if it does not contain dipoles. It is proved in \cite[Theorem
3.17]{GS} that every diagram is equivalent to a unique reduced diagram.

Now let $\P=\{c_1,c_2,\ldots\}$ be a collection of cells. The diagram group $\DG(\P,u)$ corresponding to the collection
of cells $\P$ and a word $u$ consists of all reduced $(u,u)$-diagrams obtained from the  cells of $\P$ and
trivial diagrams by using the three operations mentioned above. The product $\Delta_1\Delta_2$ of two  diagrams
$\Delta_1$ and $\Delta_2$ is the reduced diagram obtained by removing all dipoles from $\Delta_1\circ\Delta_2$. The
fact that $\DG(\P,u)$ is a group is proved in \cite{GS}.

\begin{lemma}[See \cite{GS}]\label{lm:F} If $X$
consists of one letter $x$ and $\P$ consists of one cell $x\to x^2$, then the group $\DG(\P,x)$ is the R. Thompson group
$F$.
\end{lemma}

Since $X$ consists of one letter $x$, we shall omit the labels of the edges of diagrams in $F$. Since all edges are oriented from left to right, we will not indicate the orientation of edges in the pictures of diagrams from $F$.

Thus in the case of the group $F$, the set of cells $\P$ consists of one cell $\pi$ of the form

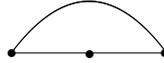
\begin{figure}[h!]
\begin{center}
\unitlength .5mm 
\linethickness{0.4pt}
\ifx\plotpoint\undefined\newsavebox{\plotpoint}\fi 
\begin{picture}(51.75,30.875)(24,20)
\put(49.25,20.25){\line(1,0){42.5}}
\qbezier(49.75,20)(70.5,47.875)(91.25,20.25)
\put(50,20.25){\circle*{2}}
\put(70.75,20){\circle*{2}}
\put(90.75,20.25){\circle*{2}}
\end{picture}
\end{center}
\caption{The cell defining the group $F$.}
\label{fd}
\end{figure}


The role of 1 in the group $F$ is played by the trivial diagram $\varepsilon(x)$ which will be denote by $\1$.

Using the addition of diagrams one can define a useful operation of addition on the group $F$: $\Delta_1\oplus \Delta_2 =\pi \circ (\Delta_1+\Delta_2)\circ \pi\iv$. Note that if $\Delta_1, \Delta_2$ are reduced, then so is $\Delta_1\oplus \Delta_2$ unless $\Delta_1=\Delta_2=\1$ in which case $\1\oplus \1=\1$.

The following property of $\oplus$ is obvious:

\begin{lemma}\label{lm0} For every $a,b,c,d\in F$ we have
\begin{equation}\label{e1}(a\oplus b)(c\oplus d)=ac\oplus bd.\end{equation}In particular $a\oplus b=(a\oplus \1)(\1\oplus b)$.
\end{lemma}

\begin{remark} \label{rk:1} Note that the sum $\oplus$ is not associative but ``almost'' associative, that is, there exists an element
$g\in F$ such that for every $a, b, c\in F$, we have $$((a\oplus b)\oplus c)^g=a\oplus (b\oplus c).$$ Figure \ref{f:xx} below shows that in fact $g=x_0$.
\end{remark}

Thus $F$ can be considered as an algebra with two binary operations: multiplication and addition. It is a group under multiplication and satisfies the identity (\ref{e1}). Algebras with two binary operations satisfying these conditions form a variety, which we shall call the variety of \emph{Thompson algebras}.
Note that every group can be turned into a Thompson algebra in a trivial way by setting $a\oplus b=1$ for every $a,b$. Our main result shows, in particular, that $F_3$ has a non-trivial structure as a Thompson algebra.

The Thompson group $F$ has an obvious involutary automorphism that flips a diagram about a vertical line (it is also an anti-automorphism with respect to addition). Thus every statement about $F$ has its left-right dual.

Note that for every $n\ge 2$, the group $F_n$ is a diagram group over the semigroup presentation $\la x \mid x=x^n\ra$ \cite{GS}. This was used in \cite{GS} to find a nice presentation of $F_n$ for every $n$. We will use this presentation below.

\subsection{A normal form of elements of $F$}

Let $x_0, x_1$ be the standard generators of $F$. Recall that $x_{i+1}, i\ge 1$,  denotes
$x_0^{-i}x_1x_0^i$.
In these generators, the group $F$ has the following presentation
$\la x_i, i\ge 0\mid x_i^{x_j}=x_{i+1} \hbox{ for every}$ $j<i\ra$ \cite{CFP}.

There exists a clear connection between representation of elements of $F$ by
diagrams and the normal form of elements in $F$. Recall~\cite{CFP} that every
element in $F$ is uniquely representable in the following form:
\begin{equation}\label{NormForm}
x_{i_1}^{s_1}\ldots x_{i_m}^{s_m}x_{j_n}^{-t_n}\ldots x_{j_1}^{-t_1},
\end{equation}
where $i_1\le\cdots\le i_m\ne j_n\ge\cdots\ge j_1$;
$s_1,\ldots,s_m,t_1,\ldots t_n\ge0$, and if $x_i$ and $x_i\iv$ occur in
(\ref{NormForm}) for some $i\ge0$ then either $x_{i+1}$ or $x_{i+1}\iv$
also occurs in~(\ref{NormForm}).
This form is called the {\em normal form} of elements in $F$.

We say that a path in a diagram is \emph{positive} if all the edges in the path are oriented from left to right. Let $\P$ be the collection of cells which consists of one cell $x\to x^2$.
It was noticed in \cite{GS1} that every reduced diagram $\Delta$ over $\P$ can be divided by its longest positive
path from its initial vertex to its terminal vertex into two parts,
{\em positive} and {\em negative}, denoted by $\Delta^+$ and $\Delta^-$,
respectively. So $\Delta=\Delta^+\circ\Delta^-$. It is easy to prove by
induction on the number of cells that  $\Delta^+$ are
$(x,x^2)$-cells and all cells in $\Delta^-$ are $(x^2,x)$-cells.

Let us show how given an $(x,x)$-diagram over $\P$ one can
get the normal form of the element represented by this diagram. This is the left-right dual of the procedure described in \cite[Example 2]{GS1} and after Theorem 5.6.41 in \cite{Sa}.

\begin{lemma}\label{lm1}
Let us number the cells of $\Delta^+$ by
numbers from $1$ to $k$ by taking every time the ``leftmost'' cell,
that is, the cell which is to the left of any other cell attached to the
bottom path of the diagram formed by the previous cells. The first cell is
attached to the top path of $\Delta^+$ (which is the top path of $\Delta$). The $i$th cell in
this sequence of cells corresponds to an edge of the Squier graph $\Gamma(\P)$,
which has the form $(x^{\ell_i},x\to x^2,x^{r_i})$, where $\ell_i$ ($r_i$)
is the length of the path from the initial (resp. terminal) vertex of the
diagram (resp. the cell) to the initial (resp. terminal) vertex of the cell
(resp. the diagram), such that the path is contained in the bottom path of the diagram formed by the first $i-1$ cells.
If $r_i=0$ then we label this cell by 1. If
$r_i\ne0$ then we label this cell by the element $x_{\ell_i}$ of $F$.
Multiplying the labels of all cells, we get the ``positive'' part of the
normal form. In order to find the ``negative'' part of the normal form, consider
$(\Delta^-)\iv$, number its cells as above and label them as above. The normal form of $\Delta$ is then the product of the normal form of $\Delta^+$ and the inverse of the normal form of $(\Delta^-)\iv$.
\end{lemma}

For example, applying the procedure from Lemma \ref{lm1} to the diagram on Figure \ref{f3}
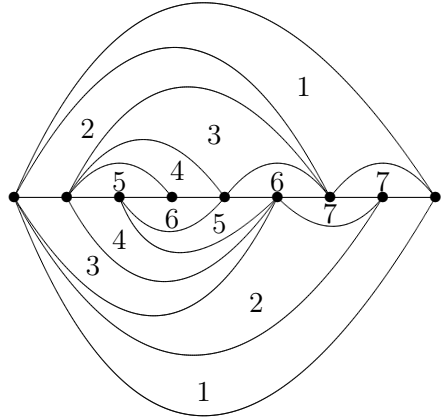
\begin{figure}[h!]
\begin{center} 
\unitlength=0.7mm
\special{em:linewidth 0.4pt}
\linethickness{0.4pt}
\begin{picture}(83.00,90.00)
\put(2.00,43.00){\circle*{2.00}}
\put(12.00,43.00){\circle*{2.00}}
\put(22.00,43.00){\circle*{2.00}}
\put(32.00,43.00){\circle*{2.00}}
\put(42.00,43.00){\circle*{2.00}}
\put(52.00,43.00){\circle*{2.00}}
\put(62.00,43.00){\circle*{2.00}}
\put(72.00,43.00){\circle*{2.00}}
\put(82.00,43.00){\circle*{2.00}}
\put(2.00,43.00){\line(1,0){10.00}}
\put(12.00,43.00){\line(1,0){10.00}}
\put(22.00,43.00){\line(1,0){10.00}}
\put(32.00,43.00){\line(1,0){10.00}}
\put(42.00,43.00){\line(1,0){10.00}}
\put(52.00,43.00){\line(1,0){10.00}}
\put(62.00,43.00){\line(1,0){10.00}}
\put(72.00,43.00){\line(1,0){10.00}}
\bezier{132}(62.00,43.00)(72.00,56.00)(82.00,43.00)
\bezier{132}(62.00,43.00)(52.00,56.00)(42.00,43.00)
\bezier{132}(12.00,43.00)(22.00,56.00)(32.00,43.00)
\bezier{212}(12.00,43.00)(26.00,65.00)(42.00,43.00)
\bezier{392}(12.00,43.00)(33.00,85.00)(62.00,43.00)
\bezier{516}(2.00,43.00)(33.00,100.00)(62.00,43.00)
\bezier{676}(2.00,43.00)(34.00,117.00)(82.00,43.00)
\bezier{132}(22.00,43.00)(31.00,30.00)(42.00,43.00)
\bezier{212}(22.00,43.00)(30.00,22.00)(52.00,43.00)
\bezier{120}(52.00,43.00)(63.00,32.00)(72.00,43.00)
\bezier{304}(12.00,43.00)(29.00,11.00)(52.00,43.00)
\bezier{412}(2.00,43.00)(30.00,-2.00)(52.00,43.00)
\bezier{556}(2.00,43.00)(35.00,-17.00)(72.00,43.00)
\bezier{740}(2.00,43.00)(34.00,-40.00)(82.00,43.00)
\put(57.00,64.00){\makebox(0,0)[cc]{1}}
\put(72.00,46.00){\makebox(0,0)[cc]{7}}
\put(16.00,56.00){\makebox(0,0)[cc]{2}}
\put(40.00,55.00){\makebox(0,0)[cc]{3}}
\put(52.00,46.00){\makebox(0,0)[cc]{6}}
\put(33.00,48.00){\makebox(0,0)[cc]{4}}
\put(22.00,46.00){\makebox(0,0)[cc]{5}}
\put(38.00,6.00){\makebox(0,0)[cc]{1}}
\put(48.00,23.00){\makebox(0,0)[cc]{2}}
\put(62.00,40.00){\makebox(0,0)[cc]{7}}
\put(17.00,30.00){\makebox(0,0)[cc]{3}}
\put(22.00,35.00){\makebox(0,0)[cc]{4}}
\put(41.00,38.00){\makebox(0,0)[cc]{5}}
\put(32.00,39.00){\makebox(0,0)[cc]{6}}
\end{picture}
\end{center}
\caption{Reading the normal form of an element of $F$ off its diagram.}
\label{f3}
\end{figure}
\noindent we get the normal form
$x_0x_1^3x_4(x_0^2x_1x_2^2x_5)\iv$.

Diagrams for the generators of $F$, $x_0, x_1$, are on Figure \ref{f:xx}. 

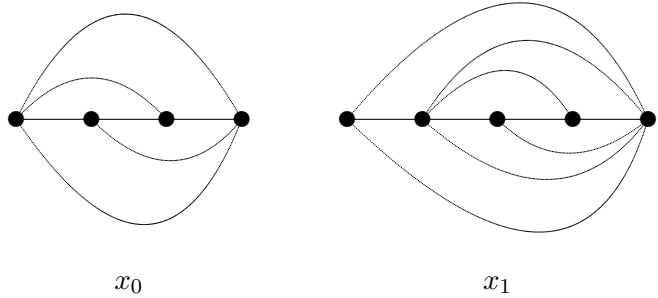
\begin{figure}[h!]
\begin{center}
\unitlength 1mm 
\linethickness{0.4pt}
\ifx\plotpoint\undefined\newsavebox{\plotpoint}\fi 
\begin{picture}(94.5,37)(0,0)
\put(39.5,24){\circle*{2}}
\put(29.5,24){\circle*{2}}
\put(19.5,24){\circle*{2}}
\put(9.5,24){\circle*{2}}
\put(93.5,24){\circle*{2}}
\put(83.5,24){\circle*{2}}
\put(73.5,24){\circle*{2}}
\put(63.5,24){\circle*{2}}
\put(53.5,24){\circle*{2}}
\put(39.5,24){\line(-1,0){10}}
\put(29.5,24){\line(-1,0){10}}
\put(19.5,24){\line(-1,0){10}}
\put(93.5,24){\line(-1,0){10}}
\put(83.5,24){\line(-1,0){10}}
\put(73.5,24){\line(-1,0){10}}
\put(63.5,24){\line(-1,0){10}}
\bezier{120}(29.5,24)(19.5,35)(9.5,24)
\bezier{120}(39.5,24)(30.5,13)(19.5,24)
\bezier{256}(39.5,24)(23.5,52)(9.5,24)
\bezier{256}(39.5,24)(28.5,-4)(9.5,24)
\bezier{132}(83.5,24)(75.5,37)(63.5,24)
\bezier{108}(93.5,24)(82.5,15)(73.5,24)
\bezier{208}(93.5,24)(76.5,45)(63.5,24)
\bezier{176}(93.5,24)(81.5,8)(63.5,24)
\bezier{296}(93.5,24)(79.5,55)(53.5,24)
\bezier{296}(93.5,24)(84.5,-6)(53.5,24)
\put(24.5,2){\makebox(0,0)[]{$x_0$}}
\put(73.5,2){\makebox(0,0)[]{$x_1$}}
\end{picture}
\end{center}
\caption{Diagrams representing the generators of the R. Thompson group $F$}
\label{f:xx}
\end{figure}

Lemma \ref{lm1} immediately implies

\begin{lemma}\label{lm1.5} If $u$ is the normal form of $\Delta$, then the normal form of $\1\oplus\Delta$ is obtained from $u$ by increasing every index by 1.
\end{lemma}

In particular, $x_1=\1\oplus x_0$, and, in general, $x_{i+1}=\1\oplus x_i$, $i\ge 0$.
Thus we get

\begin{proposition}\label{p1} As a Thompson algebra, $F$ is generated by one element $x_0$.
\end{proposition}

\subsection{From diagrams to homeomorphisms}\label{frac}

There is a natural isomorphism between Thompson group $F$ defined as a diagram group and $F$ defined as a group of homeomorphisms. Let $\Delta$ be a diagram in $F$. The positive subdiagram $\D^+$ describes a binary subdivision of $[0,1]$ in the following way. Every edge of $\Delta^+$ corresponds to a dyadic sub-interval of $[0,1]$; that is, an interval of the form $[\frac{k}{2^n},\frac{k+1}{2^n}]$ for integers $n\ge 0$ and $k=0,\dots,2^n-1$. The top edge $\Top$ corresponds to the interval $[0,1]$. For each cell $\pi$ of $\D^+$ (hence an $(x,x^2)$-cell), if $\mathbf{top}(\pi)$ corresponds to an interval $[\frac{k}{2^n},\frac{k+1}{2^n}]$ then the left bottom edge of $\pi$ corresponds to the left half of the interval, $[\frac{k}{2^n},\frac{k}{2^n}+\frac{1}{2^{n+1}}]$, and the right bottom edge of $\pi$ corresponds to the right half of the interval,  $[\frac{k}{2^n}+\frac{1}{2^{n+1}},\frac{k+1}{2^n}]$.

Thus, if the bottom path of $\D^+$ consists of $n$ edges, $\D^+$ describes a binary subdivision composed of $n$ intervals. Similarly, the negative subdiagram $\D^-$ describes a binary subdivision with $n$ intervals as well. The diagram $\D$ corresponds to the dyadic rearrangement mapping the \emph{top subdivision} (associated with $\D^+$) to the \emph{bottom subdivision} (associated with $\D^{-}$).

Thinking of $\D^+$ as a subdivision of $[0,1]$, the inner vertices of $\D$ (that is, the vertices other than $\ini$ and $\ter$) correspond to the break points of the subdivision. Thus, in the diagram $\D$ every inner vertex is associated with two break points; one break point of the top subdivision and one break point of the bottom subdivision. 

\subsection{From diagrams to pairs of binary trees}\label{trees}

Thompson group $F$ can be defined in terms of reduced pairs of binary trees.
Let $\Delta$ be a diagram in $F$ with $n+1$ vertices. 
It is possible to put a vertex in the middle of every edge of the diagram. Then, for each cell $\pi$ in $\D^+$ (hence an $(x,x^2)$-cell) one can draw edges from the vertex on $\mathbf{top}(\pi)$ to the vertices on the bottom edges of $\pi$. We get a binary tree $T_+$ with the root lying on $\Top$ and $n$ leaves lying on $\mathbf{bot}(\D^+)=\mathbf{top}(\D^-)$. A similar construction in $\D^-$ gives a second binary tree $T_-$, lying upside down, such that the leaves of $T_+$ and $T_-$ coincide. Whenever we speak of a pair of binary trees $(T_+,T_-)$ we assume that they have the same number of leaves, and that $T_-$ is drawn upside down so that the leaves of $T_+$ and $T_-$ coincide.

Let $T$ be a binary tree. We call a vertex with two children a \emph{caret} in the tree. We say that a pair of binary trees $(T_+,T_-)$ \emph{have a common caret} if for some $i$, the $i$ and $i+1$ leaves have a common father both in $T_+$ and in $T_-$ .
We call a pair of trees $(T_+,T_-)$ \emph{reduced} if it has no common carets. The construction above maps every reduced diagram $\Delta$ to a reduced pair of binary trees. The correspondence is one to one and enables to view Thompson group $F$ as a group of reduced pairs of binary trees with the proper multiplication.

\section{The Thompson graphs}\label{sec:graph}

\subsection{The definition in terms of pairs of trees}

Jones \cite{Jo} defined for each element of $F$, viewed as a reduced pair of binary trees, an associated graph, which he called the {\em Thompson graph} of that element.
If $(T_+,T_-)$ is a reduced pair of binary trees, we call an edge $t$ in $T_+$ or $T_-$ a \emph{left edge} if it connects a vertex to its left child. 

\begin{definition}[Jones \cite{Jo}]\label{def:j}
Let $g$ be an element of $F$ and $(T_+,T_-)$ the corresponding reduced pair of binary trees. If $T_+,T_-$ have $n$ common leaves, enumerated $l_1,\dots, l_n$ from left to right, we can assume that all of them lie on the same horizontal line. We define a graph $J(g)$ as follows. The graph $J(g)$ has $n$ vertices $v_0,\dots, v_{n-1}$. The vertex $v_0$ lies to the left of the first leaf $l_1$ and for all $i=1,\dots,n-1$, the vertex $v_i$ lies between $l_i$ and $l_{i+1}$ on the horizontal line.
The edges of $J(g)$ are defined as follows. For every left edge $t$ of $(T_+,T_-)$, we draw a single edge $e$ in $J(g)$ which crosses the edge $t$ and no other edge of $(T_+,T_-)$. Note that this property determines the end vertices of the edge $e$.
\end{definition}

For example, the graph $J(x_1)$, associated with $x_1$,  is depicted in Figure \ref{fig:Jx1}.

\begin{figure}[h!]
\centering
\includegraphics[width=0.5\columnwidth]{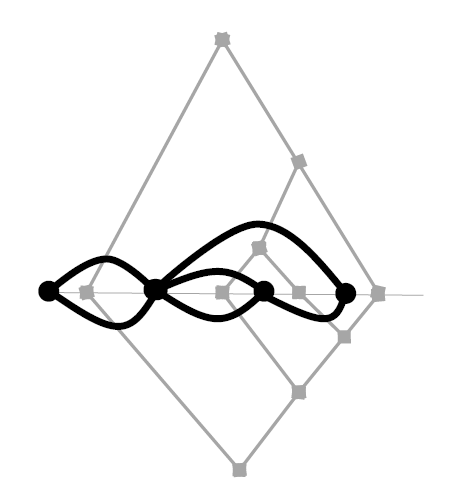}
\caption{The graph $J(x_1)$.}
    \label{fig:Jx1}
\end{figure}

\subsection{The definition in terms of diagrams}

\begin{definition}\label{def:d}
Let $\Delta$ be a (not necessarily reduced) diagram over the presentation $\P=\la x\mid x=x^2\ra$.
The \emph{Thompson graph} $T(\D)$ is a ``subgraph'' of the diagram $\D$, defined as follows. 
The vertex set of $T(\D)$ is the vertex set of $\D$ minus the terminal vertex $\tau(\D)$.
For every inner vertex $v$ of $\D$ the only incoming edges of $v$ which belong to $T(\D)$ are the top-most and bottom-most incoming edges of $v$ in $\D$.
If the top-most and bottom-most incoming edges of $v$ coincide we shall consider them as two distinct edges (hence the quotation marks on the word subgraph). An edge of $T(\D)$ will be called an \emph{upper (lower) edge} if it is the top-most (bottom-most) incoming edge of an inner vertex in $\D$. 
\end{definition}

The subgraph $T(x_0x_1)$ of the diagram $x_0x_1$ of Thompson group $F$ is depicted in Figure \ref{fig:x0x1}. The upper edges in the graph are colored red, while the lower edges are colored blue. Note that the graph is bipartite.

\begin{figure}[h]
\centering
\includegraphics[width=0.5\columnwidth]{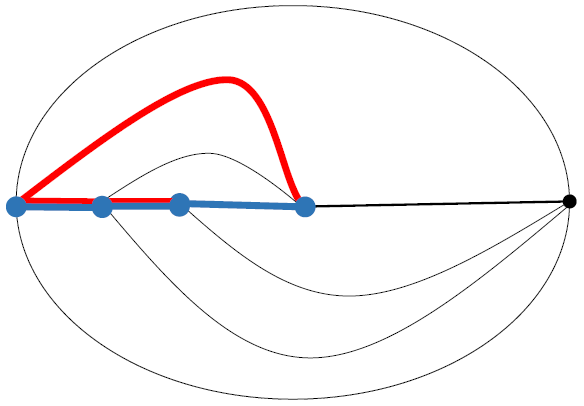}
\caption{The graph $T(x_0x_1)$.}
    \label{fig:x0x1}
\end{figure}

\begin{remark}\label{left}
Let $\Delta$ be a reduced diagram in $F$. Consider the positive subdiagram $\D^+$. Every cell $\pi$ in $\D^+$ is an $(x,x^2)$-cell. As such, it has a unique \emph{bottom vertex} separating the left bottom edge of $\pi$ from its right bottom edge (see Figure \ref{fd}).  
Conversely, every inner vertex $v$ of $\D$ is the bottom vertex of a unique cell $\pi_v$ in $\D^+$. Clearly, the top-most incoming edge of $v$ in $\D$ is the left bottom edge of the cell $\pi_v$. Thus, the upper edges in the graph $T(\D)$ are exactly the left bottom edges of the cells in $\D^+$. Similarly, the lower edges of the graph $T(\D)$ are exactly the left top edges of the cells in $\D^-$.
\end{remark}

\begin{remark} In Section \ref{ss:tf} we shall show how to reconstruct $g$ from $T(g)$.
\end{remark}

\subsection{The equivalence of two definitions}

\begin{proposition}
Let $\Delta$ be a reduced diagram in $F$. Then, the associated graphs $T(\Delta)$ (from Definition \ref{def:d}) and $J(\Delta)$ (from Definition \ref{def:j}) are isomorphic.
\end{proposition}

\begin{proof}
Let $(T_+,T_-)$ be the reduced pair of binary trees associated with $\Delta$. We can assume that the pair $(T_+,T_-)$ is drawn inside the diagram $\D$ as described in Section \ref{trees}.
Let $T(\D)$ be the subgraph associated with $\D$. It is possible to stretch each of the upper edges of $T(\D)$ up as follows.
By Remark \ref{left}, $e$ is an upper edge in $T(\D)$, if and only if $e$ is a left bottom edge of some cell in $\D^+$. Let $v$ be the vertex of $T_+$ which lies on $e$ and $t$ the edge of $T_+$ connecting $v$ to its father. Clearly, $t$ is a left edge in the tree $T_+$. We stretch the edge $e$ slightly so that instead of crossing the vertex $v$ it crosses the edge $t$ of $T_+$ (and no other edges of the tree).
Similarly, we stretch every bottom edge of $T(\D)$ down so that instead of crossing a vertex of the tree $T_-$, it crosses the edge of the tree connecting the vertex to its father. The process is illustrated in Figure \ref{fig:stretch}.

\begin{figure}[h]
\centering
\includegraphics[width=0.8\columnwidth]{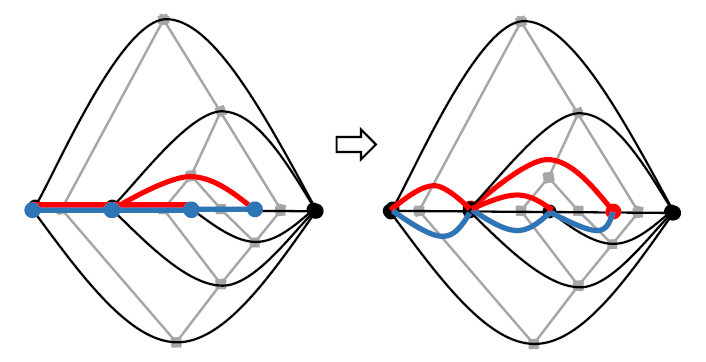}
\caption{Stretching the edges of $T(\D)$.}
    \label{fig:stretch}
\end{figure}

If the graph $T'(\D)$ results from stretching the edges of $T(\D)$ as described, then there is a one to one correspondence between the left edges of the pair of trees $(T_+,T_-)$ and the edges of $T'(\D)$. Indeed, every edge of $T'(\D)$ crosses a single left edge of $(T_+,T_-)$ and every left edge of $(T_+,T_-)$  is crossed by an edge of $T'(\D)$. Note, that if $(T_+,T_-)$ have $n$ common leaves, then $T(\D)$ (hence $T'(\D)$) has $n$ vertices; one to the left of the left most leaf of $T_+$ and one between any pair of consecutive leaves of $T_+$. It follows that the graph $T'(\D)$ is isomorphic to the graph $J(\D)$. Since $T(\D)$ and $T'(\D)$ are clearly isomorphic as graphs we get the result.
\end{proof}

\section{The Jones' subgroup $\protect\overrightarrow F$ and its properties}\label{sec:properties}

\subsection{The definition of $\protect\overrightarrow F$}

\begin{lemma}\label{l:1}
Suppose that a diagram $\Delta$ is obtained from a diagram $\Delta'$ by removing a dipole. Suppose that $T(\Delta')$ is bipartite.  Then $T(\Delta)$ is bipartite.
\end{lemma}

\begin{proof}
The dipole can be of type $\pi\circ \pi\iv$ or of type $\pi\iv\circ \pi$ where $\pi$ is the cell on Figure \ref{fd}.
 In the first case to get the graph $T(\Delta)$ we remove from $T(\D')$ a vertex with exactly two edges connecting it to another vertex of $T(\Delta')$, and the statement is obvious.
In the second case, since $T(\Delta')$ is bipartite, we can label the vertices of $\Delta'$ by "+" and "-", so that every two incident vertices have opposite signs. Consider the four vertices of the dipole. In $T(\Delta')$ the top and the bottom vertices of the dipole are incident to the left vertex. Hence the top and the bottom vertices have the same label. Note that the edge of the dipole connecting the left vertex with the right vertex is not an edge of $T(\Delta')$. Thus, the effect of removing the dipole on $T(\Delta')$ amounts to identifying the top and the bottom vertices of the dipole and erasing the lower edge connecting the left vertex with the top vertex and the upper edge connecting the left vertex with the bottom one. Since the top and bottom  vertices have the same label, the Thompson graph $T(\Delta)$ is bipartite.
\end{proof}

\begin{definition}\label{vecF}
\emph{Jones' subgroup} $\overrightarrow F$ is the set of all reduced diagrams $\D$ in $F$ for which the associated graph $T(\D)$ is bipartite.
\end{definition}

\begin{proposition}\label{sub}
Jones' subgroup $\overrightarrow F$ is indeed a subgroup of $F$.
\end{proposition}

\begin{proof}
Suppose that $\Delta_1, \Delta_2$ belong to $\overrightarrow F$. The Thompson graph $T(\Delta_1\circ \Delta_2)$ is the union of $T(\Delta_1)$ and $T(\Delta_2)$ with vertices $\iota(\Delta_1)$ and $\iota(\Delta_2)$ identified. Hence $T(\Delta_1\circ\Delta_2)$ is bipartite. By Lemma \ref{l:1}, the graph
$T(\Delta_1\Delta_2)$ is bipartite as well.
\end{proof}

\subsection{The subgroup $\protect\overrightarrow F$ is isomorphic to $F_3$}


\begin{lemma}\label{lm2}
The Jones' subgroup $\overrightarrow F$  coincides with the subgroup $H$ which is the smallest subgroup of $F$ that contains $x_0x_1$ and closed under addition. \end{lemma}

\proof
Clearly $H$ is inside $\overrightarrow F$. Also from Lemma \ref{lm1} it follows that
if we add the trivial diagram $\1$ on the right
to the reduced diagram representing $x_0x_1$, we get the diagram corresponding to the normal form
$x_0x_0x_1 (x_0x_1x_2)^{-1}=x_0x_0x_1x_2^{-1}(x_0x_1)^{-1}$. Hence
$x_0x_0x_1x_2^{-1}$ also belongs to $H$. If we add to this element the
diagram $\1$ on the right, we get $(x_0x_0x_0x_1)(x_0x_1x_2x_2)^{-1}$,
and so  the element $x_0^3x_1x_2^{-2}$ belongs to $H$. By induction, we see
that all elements $x_0^nx_1x_2^{-n+1}$ belong to $H$. Now consider an
arbitrary reduced diagram $\Delta$ in $\overrightarrow F$. Let us enumerate the vertices of $\Delta$ from left to right: $0,1,...,s$ so that $\iota(\Delta)=0, \tau(\Delta)=s$.

If the Thompson graph $T(\Delta$) does not contain
top nor bottom edges connecting 
$j>1$
with $0$, then $\Delta$ is a sum
of $\Delta'$ and the trivial diagram $\1$ (added on the left). The diagram $\Delta'$
also belongs to $\overrightarrow F$ (its graph $T(\Delta')$ is bipartite), so by induction $\Delta'$
is in $H$, and $\Delta$ is in $H$ also. So assume that $T(\Delta)$ contains an edge
$(0,j), j>1$. 
Without loss of generality we can assume that this is an upper edge, that is, it belongs to the  positive part of
the diagram. Therefore the positive part of the normal form for $\Delta$ in
the generators $x_0, x_1, x_2,...$ starts with $x_0$. Let it start with
$x_0^nx_i, i>0$. The bottom-most cell corresponding to the prefix $x_0^n$ cannot
have both its bottom edges on the maximal positive path of $\Delta$, i.e. its
top edge cannot connect $0$ with $2$. Indeed if it connects
$0$ with $2$, then $2$ and $0$ are in
different parts of the bipartite graph $T(\Delta)$. The vertex $1$
then belongs to the same part as $2$. Then $2$ cannot be connected with $1$ by a lower edge from $T(\Delta)$, so it has to be connected with $0$, which means that the diagram $\Delta$
is not reduced, a contradiction.
Therefore there is an edge connecting $1$ with
$j'\le j$. Hence the normal form corresponding to $\Delta$ starts
with $x_0^nx_1$. If $n=1$, then we can divide by $x_0x_1$ (on the left) and
get an element from $\overrightarrow F$ with a shorter normal form. Hence $\Delta$ is in $H$
since $x_0x_1$ is in $H$. If $n>1$, then  we can replace $x_0^nx_1$ by
$x_2^{n-1}$, and the resulting element would still be in $\overrightarrow F$. Since its
normal form is shorter than that of $\Delta$, it is in $H$, and so $\Delta$ is in $H$.
\endproof

The proof of Lemma \ref{lm2} proves the following.

\begin{lemma}\label{lm3'} The subgroup $\overrightarrow F$ is generated by two
sets $X=\{x_ix_{i+1}, i\ge 0\}$ and $X'=
\{x_i^{n+1}x_{i+1}x_{i+2}^{-n}, i\ge 0, n\ge 1\}$.
\end{lemma}

\proof Indeed, in the proof of Lemma \ref{lm2} we proved that $X$ and $X'$ are inside $\overrightarrow F$, and every element of $\overrightarrow F$ is either a sum of the trivial diagram $\1$ and some other diagram from $\overrightarrow F$ or is a product of an element from $X\cup X'$ and a reduced diagram from $\overrightarrow F$ with a shorter normal form or is the inverse of such an element. Since by Lemma \ref{lm1.5} the subgroup generated by $X\cup X'$ is closed under addition of the trivial diagram $\1$ on the left, the lemma is proved. \endproof

\begin{lemma}\label{lm3} The subgroup $\overrightarrow F$ is generated by three elements $x_0x_1, x_1x_2, x_2x_3$.
\end{lemma}

\proof Let $X$ and $X'$ be the sets from Lemma \ref{lm3'}. It is obvious that for every $j>2$ the element $x_jx_{j+1}$ is equal to $(x_{j-2}x_{j-1})^{x_0x_1}$. Thus the set $X=\{x_jx_{j+1}, j\ge 0\}$ is contained in the subgroup $\la x_0x_1, x_1x_2, x_2x_3\ra$.
It remains to show that $X'\subseteq \la X\ra$. Note that
\begin{equation}
\begin{split}
\nonumber x_0^2x_1x_2^{-1} & = x_0x_1x_1^{-1}x_0x_1x_2^{-1} = x_0x_1x_0x_2^{-1}x_1x_2^{-1} =
x_0x_1x_0x_1x_3^{-1}x_2^{-1}  \\
& = (x_0x_1)^2(x_2x_3)^{-1}.
\end{split}
\end{equation}

 Similarly,
$$x_0^{n+1}x_1x_2^{-n}=(x_0x_1)^{n+1}(x_2x_3)^{-n}$$ for every $n\ge 1$. Adding several trivial diagrams $\1$ on the left, we get
$$x_j^{n+1}x_{j+1}x_{j+2}^{-n}=(x_jx_{j+1})^{n+1}(x_{j+2}x_{j+3})^{-n}$$ for every $j\ge 0$.
\endproof

\begin{lemma}\label{lm4} The subgroup $\overrightarrow F$ is isomorphic to $F_3$.
\end{lemma}
\proof  The elements $x_0x_1, x_1x_2, x_2x_3$ satisfy the defining
relations of $F_3$ (see \cite[page 54]{GS}). All proper homomorphic images of $F_3$ are Abelian \cite[Theorem 4.13]{B}. Since $x_0x_1, x_1x_2$
do not commute, the natural homomorphism from $F_3$ onto $\overrightarrow F$ is an isomorphism.
\endproof

As an immediate corollary of Theorem \ref{thm:1}, we get

\begin{proposition}[Compare with Proposition \ref{p1}] \label{p2} $\overrightarrow F$ is a subalgebra of the Thompson algebra $F$ generated by one element $x_0x_1$.
\end{proposition}

\subsection{$\protect\overrightarrow F$ is the stabilizer of the set of dyadic fractions with odd sums of digits}\label{sec:stab}

Let $\D$ be a reduced diagram in $F$ and $T(\D)$ the associated graph. Let $\D^+$ be the positive subdiagram. 
%
Recall (Section \ref{frac}) that $\D^+$ describes a binary subdivision of $[0,1]$. Every edge $e$ in $\D^+$ corresponds to a dyadic interval of length $\frac{1}{2^m}$ for some integer $m\ge 0$. We will call this length the \emph{weight of the edge $e$} and denote it by $\omega(e)$. The inner vertices of $\D^+$ correspond to the break points of the subdivision. Note that if $v$ is an inner vertex of $\D^+$, the dyadic fraction $f$ is the corresponding break point of the subdivision and $p$ is a positive path in $\D^+$ from $\ini$ to $v$, then the weight of the path $p$ (that is, the sum of weights of its edges) is equal to the fraction $f$.

\begin{lemma}\label{<}
Let $v$ be an inner vertex of $\D$. Let $e$ be the (unique) upper incoming edge of $v$ in $T(\D)$ and $e_1$ be any outgoing upper edge of $v$ in $T(\D)$, then $\omega(e_1)<\omega(e)$.
\end{lemma}

\begin{proof}
By Remark \ref{left}, $e$ is the left bottom edge of some cell $\pi$ in $\D^+$. Let $e'$ be the right bottom edge of the cell $\pi$. Clearly, the edge $e'$ is the top-most outgoing edge of $v$ in $\D$. Since $e'$ is not the left bottom edge of any cell in $\D^+$, by Remark \ref{left}, the edge $e'$ does not belong to $T(\D)$. Hence, the edge $e_1\neq e'$ so that $e_1$ lies under $e'$ in $\D^+$. From the construction in Section \ref{frac} it is obvious that $\omega(e_1)\le \frac{1}{2}\omega(e')=\frac{1}{2}\omega(e)<\omega(e)$.
\end{proof}

\begin{definition}
Let $\D$ be a (not necessarily reduced) diagram. Let $v$ be an inner vertex of $\D$. Since every inner vertex of $\D$ has a unique incoming upper edge in $T(\D)$, there is a unique positive path in $T(\D)$ from $\ini$ to $v$, composed entirely of upper edges. We call this path the \emph{top path from $\ini$ to $v$}.
\end{definition}

\begin{lemma}\label{digits}
Let $\D$ be a reduced diagram in $F$. Let $v$ be an inner vertex of $\D$ and $f$ the corresponding break point of the subdivision associated with $\D^+$. Let $p$ be the top path from $\ini$ to $v$. Then, the length of $p$ is equal to the sum of digits in the binary form of $f$.
\end{lemma}

\begin{proof}
Let $e_1,\dots,e_k$ be the edges of $p$. For each $i=1,\dots k$, the weight $\omega(e_i)=\frac{1}{2^{m_i}}$ for some positive integer $m_i$. By Lemma \ref{<}, $\omega(e_1)>\dots>\omega(e_k)$. Thus, the weight of $p$ is $\omega(p)=\sum_{i=1}^k{\frac{1}{2^{m_i}}}$
where  $\frac{1}{2^{m_1}}>\dots>\frac{1}{2^{m_k}}$. Clearly, the sum of digits of $\omega(p)$ in binary form is $k$. Therefore, the sum of digits of $f=\omega(p)$ is equal to the number of edges in $p$.
\end{proof}

Viewed as a dyadic rearrangement, a reduced diagram $\D$ takes the break points of the top subdivision (associated with $\D^+$) to the break points of the bottom subdivision (associated with $\D^-$). Reflecting the diagram $\D$ about a horizontal line shows that the analogue of Lemma \ref{digits} holds for the bottom subdivision: If $v$ is an inner vertex of $\D$ and $p$ is the \emph{bottom path from $\ini$ to $v$}, composed entirely of lower edges of $T(\D)$, then the length of $p$ is equal to the sum of digits of the  break point corresponding to $v$ in the bottom subdivision.

\begin{cor}\label{bipartite}
Let $S$ be the set of dyadic fractions with odd sums of digits. Let $\D$ be a reduced diagram which stabilizes $S$ and $v$ be an inner vertex of $\D$. Then, the length of the top path from $\ini$ to $v$ and the length of the bottom path from $\ini$ to $v$ have the same parity.
\end{cor}

\begin{proof}
Let $f^+,f^-$ be the break points corresponding to $v$ in the top and bottom subdivisions respectively. Since $\D$ takes $f^+$ to $f^-$, the sum of digits of $f^+$ and the sum of digits of $f^-$ have the same parity. The result follows from Lemma \ref{digits} and its stated analogue.
\end{proof}

Now we are ready to prove

\begin{Theorem2}
Jones' subgroup $\overrightarrow F$ is the stabilizer of the set of dyadic fractions from the unit interval $[0,1]$ with odd sums of digits, under the standard action of $F$ on the interval $[0,1]$.
\end{Theorem2}

\begin{proof}
Let $S$ be the set of dyadic fractions with odd sums of digits. Let $\Delta$ be a reduced diagram which stabilizes $S$.
By Corollary \ref{bipartite}, for every inner vertex $v$, the length of the top path from $\ini$ to $v$ and the length of the bottom path from $\ini$ to $v$ have the same parity.
It is possible to use this parity to assign to every vertex of $T(\D)$ a label "0" or "1". Since for every vertex there is a unique top path and a unique bottom path from $\ini$ to the vertex, neighbors in $T(\D)$ have different labels and $T(\D)$ is bipartite.

The other direction follows from Theorem \ref{thm:1}.
On finite binary fractions $x_0x_1$ acts as follows:
\[
 x_0x_1(t) =
  \begin{cases}
   t=.00\alpha & \rightarrow t=.0\alpha  \\
   t=.010\alpha       & \rightarrow  t=.10\alpha \\
   t=.011\alpha       & \rightarrow t=.110\alpha \\
   t=.1\alpha       & \rightarrow t=.111\alpha 	
  \end{cases}
\]

In particular, $x_0x_1$ stabilizes the set $S$.
If $\Delta$ and $\D'$ are diagrams in $F$, then viewed as maps from $[0,1]$ to itself, the sum $\D\oplus\D'$ is defined as
\[(\Delta\oplus\Delta')(t) = \left\{
  \begin{array}{lr}
    \frac{\Delta(2t)}{2} & : t\in [0,\frac{1}{2}]\\
    \frac{\Delta'(2t-1)}{2}+\frac{1}{2} & : t\in [\frac{1}{2},1]
  \end{array}
\right.
\]
It is easy to see that if $\Delta$ and $\Delta'$ stabilize $S$, then $\D\oplus\D'$ stabilize $S$ as well.
Since by Theorem \ref{thm:1}, $\overrightarrow F$ is the smallest subgroup of $F$ containing $x_0x_1$ and closed under sums, the inclusion $\overrightarrow F\subset \Stab (S)$ follows.
\end{proof}

In order to prove Corollary \ref{thm:3} we will need the following observations.

\begin{remark}\label{cor}
Let $c=(x_0x_1)^{-1}\in F$. On finite binary fractions $c$ acts as follows:
\[
 c(t) =
  \begin{cases}
	t=.0\alpha\ & \rightarrow\ t=.00\alpha\\
	t=.10\alpha\ & \rightarrow\ t=.010\alpha\\
	t=.110\alpha\ & \rightarrow\ t=.011\alpha\\
	t=.111\alpha\ & \rightarrow\ t=.1\alpha
  \end{cases}
\]
In particular, if $t$ is a finite binary fraction and $m\in\mathbb{N}$ then for any large enough $n\in\mathbb{N}$, the first $m$ digits in the binary form of $c^n(t)$ are zeros. That is, $c^n(t)<\frac{1}{2^m}$.
\end{remark}

\begin{proof}
If the first digit of $t$ is $0$ then each application of $c$ adds another $0$ to the leading sequence of zeros in the binary form of $t$ so for every $n\ge m$ we get the result.
If the first digit of $t$ is $1$, then $t$ starts with a sequence of ones followed by a $0$. Let $l$ be the length of this sequence. If $l>3$ then each application of $c$ reduces the length of the sequence of ones by $2$. Thus, possibly after several applications of $c$, we can assume that $l=1$ or $l=2$.  In both cases, one application of $c$ yields $c(t)$ which starts with $0$ and we are done by the previous case.
\end{proof}

\begin{lemma}\label{remark}
Let $g$ be an element of R. Thompson group $F$. Then, there exists $m\in\mathbb{N}$ such that for any finite binary fraction $t<\frac{1}{2^m}$ the sum of digits of $t$ in binary form is equal to the sum of digits of $g(t)$.
\end{lemma}

\begin{proof}
The element $g$ maps some binary subdivision $B_1$ onto a subdivision $B_2$. The first segment of $B_1$ is of the form $J=[0,\frac{1}{2^{r}}]$ for some positive integer $r$. Since $0$ is mapped to $0$, on the segment $J$ the function $g$ is defined as a linear function with slope $2^l$ for some $l\in\mathbb{Z}$ and with constant number $0$. That is,
for every $t\in J$ we have $g(t)=2^lt$.
If $l\le 0$, then for any binary fraction $t\le \frac{1}{2^{r}}$, the application of $g$ adds $l$ zeros to the beginning of the binary form of $g$, which does not affect the sum of digits of $t$. In that case, taking $m=r$ would do.
If $l>0$ it is possible to take $m=\max\{r,l\}$. Since $m\ge r$, the binary fraction $t\in J$. Since $m\ge l$, the binary form of $t$ starts with at least $l$ zeros. The application of $g$ erases the first $l$ zeros and thus does not affect the sum of digits of $t$.
\end{proof}

Corollary \ref{thm:3} follows immediately from the following.

\begin{theorem}\label{question 2}
Let $h$ be an element of $F$ which does not belong to $\overrightarrow F$. Then the index $[\overrightarrow F:\overrightarrow F\cap h\overrightarrow Fh^{-1}]$ is infinite.
\end{theorem}

\begin{proof}
Let $h\notin \overrightarrow F$. If the index $[\overrightarrow F:\overrightarrow F\cap h\overrightarrow Fh^{-1}]$ is finite then there exists $r\in\mathbb{N}$ such that for every $g\in \overrightarrow F$, we have $g^r\in h\overrightarrow Fh^{-1}$. That is, $h^{-1}g^rh\in \overrightarrow F$. In particular, for every $k\in \mathbb{N}$ we have $h^{-1}g^{rk}h\in \overrightarrow F$.
Let $g=(x_0x_1)^{-1}\in \overrightarrow F$. We will show that for every $n$ large enough, $h^{-1}g^{n}h\notin \overrightarrow F$ and get the required contradiction.

Let $S$ be the set of finite binary fractions with odd sums of digits. Since $h^{-1}\notin \overrightarrow F$, there exists $t\in S$ such that $h^{-1}(t)\notin S$. Let $t_1=h^{-1}(t)$.
By Lemma \ref{remark} there exists $m$ for which the sum of digits of every binary fraction $<\frac{1}{2^m}$ is preserved by $h$. By Remark \ref{cor} for every $n$ large enough, $g^n(t_1)<\frac{1}{2^m}$. Since $g^n\in \overrightarrow F$ and $t_1\notin S$, the binary fraction $g^n(t_1)\notin S$. Since $g^n(t_1)<\frac{1}{2^m}$, the sum of digits of $h(g^n(t_1))$ is equal to the sum of digits of $g^n(t_1)$. Thus $h(g^n(t_1))\notin S$.
Therefore, (recall that the composition in $F$ is from left to right), $h^{-1}g^{n}h(t)=h(g^n(h^{-1}(t)))=h(g^n(t_1))\notin S$. The element $t$ belonging to $S$ implies that $h^{-1}g^{n}h$ does not stabilize $S$ and in particular $h^{-1}g^{n}h\notin \overrightarrow F$.

\end{proof}

%
%
%
%

\section{The subgroup $\protect\overrightarrow{F_n}$}\label{sec:5}

In this section we generalize results of the previous sections from $2$ and $3$  to arbitrary $n$. It turns out that the generalization is quite natural. The proofs follow the same paths and for some theorems, the proof for arbitrary $n$ is almost identical to the proof of the particular case considered before.

\subsection{The definition of the subgroup}

\begin{definition}
Let $\Delta$ be a (not necessarily reduced) diagram over the presentation $\P=\la x\mid x=x^2\ra$ and let $n\in\mathbb{N}$.
The diagram $\D$ is said to be \emph{$n$-good} if for every inner vertex $v$ the lengths of the top and bottom paths from $\ini$ to $v$ are equal modulo $n$.
\end{definition}

Note that if $n=2$, then being $2$-good is formally weaker than being bipartite. Lemma \ref{lm2n} shows that these conditions are in fact equivalent for reduced diagrams.  

The proof of the following lemma is completely analogous to the proof of Lemma \ref{l:1}, so we leave the proof to the reader.

\begin{lemma}\label{l:1n}
Suppose that a diagram $\Delta$ is obtained from a diagram $\Delta'$ by removing a dipole. Suppose that $\Delta'$ is $n$-good for some $n\in\mathbb{N}$. Then $T(\Delta)$ is $n$-good as well.
\end{lemma}

\begin{definition}\label{vecFn}
Let $n\in\mathbb{N}$. \emph{Jones' $n$-subgroup} $\overrightarrow{F_n}$ is the set of all reduced $n$-good diagrams $\D$ in $F$.
\end{definition}

In particular Jones' $1$-subgroup is the entire Thompson group $F$. Jones' $2$-subgroup coincides with Jones' subgroup $\overrightarrow F$.

The proof of the following proposition is identical to that of Proposition \ref{sub}.

\begin{proposition}
For every $n\in\mathbb{N}$, Jones' $n$-subgroup $\overrightarrow{F_n}$ is indeed a subgroup of $F$.
\end{proposition}


\subsection{The subgroup $\protect\overrightarrow{F}_{n-1}$ is isomorphic to $F_n$}

Let $n\ge 2$. We denote by $H_n$ the Thompson subalgebra of $F$ generated by $x_0\cdots x_{n-2}$, i.e., the smallest subgroup of $F$ containing $x_0\cdots x_{n-2}$ and closed under $\oplus$.

\begin{lemma}\label{technic}
Let $i,d\in\mathbb{N}\cup\{0\}$. Let $(m_0,\dots,m_d)$ be a sequence of positive integers such that $m_d\ge 2$ if $d>0$. Then,
$$\prod_{k=0}^{d}{x_{i+k}^{m_k}}\prod_{k=1}^{n-2}{x_{i+d+k}}\left[\left(\prod_{k=0}^{d-1}{x_{i+n-1+k}^{m_k}}\right)x_{i+n-1+d}^{m_d-1}\right]\iv \in H_n.$$
\end{lemma}

\begin{proof}
We first prove the Lemma for $i=0$ by induction on $d$.
For $d=0$, the argument is similar to the one in the proof of Lemma \ref{lm2}.
If we add the trivial diagram $\1$ on the right
to the reduced diagram representing $\prod_{k=0}^{n-2}{x_k}$, we get the diagram corresponding to the normal form
$x_0^2 \prod_{k=1}^{n-2}{x_k} (\prod_{k=0}^{n-1}{x_k})^{-1}=x_0^2 (\prod_{k=1}^{n-2}{x_k}) x_{n-1}^{-1}(\prod_{k=0}^{n-2}{x_k})^{-1}$.
Hence $x_0^2 (\prod_{k=1}^{n-2}{x_k}) (x_{n-1})^{-1}$ also belongs to $H_n$. If we add to this element the
diagram $\1$ on the right, we get $x_0^3 \prod_{k=1}^{n-2}{x_k} [(\prod_{k=0}^{n-1}{x_k})x_{n-1}]^{-1}=x_0^3 (\prod_{k=1}^{n-2}{x_k}) (x_{n-1}^{2})\iv(\prod_{k=0}^{n-2}{x_k})^{-1}$.
Multiplying on the right by $\prod_{k=0}^{n-2}{x_k}$ we get that $x_0^3 (\prod_{k=1}^{n-2}{x_k}) (x_{n-1}^2)^{-1}$ belongs to $H_n$.

Note the effect of the addition of $\1$ on the right to the element $x_0^2 (\prod_{k=1}^{n-2}{x_k}) (x_{n-1})^{-1}$. The negative part of the normal form got multiplied on the left by $\prod_{k=0}^{n-1}{x_k}$, when originally it started with $x_{n-1}$. Similarly, the positive part of the normal form was multiplied by $x_0$, when originally it started with $x_0$. In particular the exponents of both $x_0$ and $x_{n-1}$ were increased by one. 
By repeating the process of adding $\1$ on the right and multiplying by $\prod_{k=0}^{n-2}{x_k}$ (on the right) we get that for every positive integer $m_0$,
the element $x_0^{m_0} (\prod_{k=1}^{n-2}{x_k}) (x_{n-1}^{m_0-1})\iv$ belongs to $H_n$ as required.

Assume that the lemma holds for every non negative integer $\le d$.
Let $(m_0,\dots,m_{d+1})$ be a sequence of positive integers such that $m_{d+1}\ge 2$.
Let $j=\min\{1,\dots,d+1\}$ such that $m_j\ge 2$. For all $r=1,\dots,d+1-j$ let $n_r=m_{r+j}$. Let $n_0=m_j-1$.
We use the induction hypothesis on $d+1-j$ with the sequence $(n_0,\dots,n_{d+1-j})$. Thus, we have that
$$\prod_{k=0}^{d+1-j}{x_k^{n_k}}\prod_{k=1}^{n-2}{x_{d+1-j+k}}\left[\left(\prod_{k=0}^{d+1-j-1}{x_{n-1+k}^{n_k}}\right)x_{n-1+d+1-j}^{n_{d+1-j}-1}\right]\iv \in H_n$$
Adding the trivial diagram $\1$ $j$ times on the left we get by Lemma \ref{lm1.5} the element
$$\prod_{k=0}^{d+1-j}{x_{k+j}^{n_k}}\prod_{k=1}^{n-2}{x_{d+1+k}}\left[\left(\prod_{k=0}^{d-j}{x_{n-1+k+j}^{n_k}}\right)x_{n+d}^{n_{d+1-j}-1}\right]\iv.$$ 
We assume that $j<d+1$, the other case being similar. Then, adding $\1$ on the right results in the following element.
Note that as in the case $d=0$ the positive part of the normal form gets multiplied by $\prod_{k=0}^j{x_k}$ (it currently starts with $x_j$). Similarly, the negative part of the normal form is multiplied by $\prod_{k=0}^{n-1+j}{x_k}$.
$$\left(\prod_{k=0}^{j-1}{x_k}\right)x_{j}^{n_0+1}\prod_{k=1}^{d+1-j}{x_{k+j}^{n_k}}\prod_{k=1}^{n-2}{x_{d+1+k}}
\left[\left(\prod_{k=0}^{n-2+j}x_k\right)x_{n-1+j}^{n_0+1}\left(\prod_{k=1}^{d-j}{x_{n-1+k+j}^{n_k}}\right)x_{n+d}^{n_{d+1-j}-1}\right]\iv$$ 
Substituting $n_k$ by $m_{k+j}$ and $n_0+1$ by $m_j$ we get
$$\left(\prod_{k=0}^{j-1}{x_k}\right)x_{j}^{m_j}\prod_{k=1}^{d+1-j}{x_{k+j}^{m_{k+j}}}\prod_{k=1}^{n-2}{x_{d+1+k}}
\left[\left(\prod_{k=0}^{n-2+j}x_k\right)x_{n-1+j}^{m_j}\left(\prod_{k=1}^{d-j}{x_{n-1+k+j}^{m_{k+j}}}\right)x_{n+d}^{m_{d+1}-1}\right]\iv$$ 
Then, shifting the indexes we get that
$$\prod_{k=0}^{j-1}{x_k}\prod_{k=j}^{d+1}{x_{k}^{m_{k}}}\prod_{k=1}^{n-2}{x_{d+1+k}}
\left[\prod_{k=0}^{n-2+j}x_k\left(\prod_{k=j}^{d}{x_{n-1+k}^{m_{k}}}\right)x_{n+d}^{m_{d+1}-1}\right]\iv \in H_n$$
Multiplying on the right by $\prod_{k=0}^{n-2}{x_k}$ cancels a prefix of the negative part of the normal form so we have that
$$\prod_{k=0}^{j-1}{x_k}\prod_{k=j}^{d+1}{x_{k}^{m_{k}}}\prod_{k=1}^{n-2}{x_{d+1+k}}
\left[\prod_{k=0}^{j-1}x_{n-1+k}\left(\prod_{k=j}^{d}{x_{n-1+k}^{m_{k}}}\right)x_{n+d}^{m_{d+1}-1}\right]\iv \in H_n$$
Since $m_k=1$ for all $k=1,\dots,j-1$ we have,
$$x_0\prod_{k=1}^{d+1}{x_{k}^{m_{k}}}\prod_{k=1}^{n-2}{x_{d+1+k}}
\left[x_{n-1}\left(\prod_{k=1}^{d}{x_{n-1+k}^{m_{k}}}\right)x_{n+d}^{m_{d+1}-1}\right]\iv \in H_n$$
To get the result for $m_0\ge 1$ it is possible to increase the exponents of $x_0$ and $x_{n-1}$ simultaneously by repeatedly adding $\1$ on the right and multiplying by $\prod_{k=0}^{n-2}{x_k}$. Thus, we have that
$$\prod_{k=0}^{d+1}{x_{k}^{m_{k}}}\prod_{k=1}^{n-2}{x_{d+1+k}}
\left[\left(\prod_{k=0}^{d}{x_{n-1+k}^{m_{k}}}\right)x_{n+d}^{m_{d+1}-1}\right]\iv \in H_n$$
as required.
If $i\neq 0$ then by Lemma \ref{lm1.5}, adding the trivial diagram $i$ times on the left, gives the result.
\end{proof}

\begin{lemma}\label{lm2n}
Let $n\ge 2$. Then, the subgroup $\overrightarrow{F}_{\!\!n-1}$ coincides with the Thompson subalgebra $H_n$.
\end{lemma}

\begin{proof}
Clearly, $H_n$ is inside $\overrightarrow{F}_{\!\!n-1}$.
Let $\Delta$ be a reduced diagram in $\arr{F}_{\!\!n-1}$. We enumerate the vertices of $\Delta$ from left to right: $0,1,...,s$ so that $\iota(\Delta)=0, \tau(\Delta)=s$. We shall need the following lemma.

\begin{lemma}
Let $r$ be the left-most vertex such that there exists $\ell<r-1$ such that
$\ell$ and $r$ are connected by an edge in $T(\Delta)$. Let $\ell$ be the left-most vertex such that $(\ell,r)$ is an upper or lower  edge in $T(\D)$. Then $r-\ell\ge n$. That is, there are at least $n-1$ inner vertices of $\D$ between the vertices $\ell$ and $r$.
\end{lemma}

\begin{proof}
We assume that the edge $(\ell,r)$ is an upper edge in $T(\D)$. Otherwise, we look at $\D\iv $ instead.  Note that a priori, there might also be a lower edge in $T(\D)$ connecting the vertices $\ell$ and $r$. Let $\D_1$ be the subdiagram of $\D^+$ bounded from above by the edge $(l,r)$ and from below by (part of) the path $\bott(\D^+)$. Since every inner vertex $i<r$ has no incoming edges in $\D$ other than $(i-1,i)$, every inner vertex of $\D_1$ is connected by an edge in $\D_1$ to the terminal vertex $r$.
It follows that the vertices $\ell$ and $r$ are not connected by a lower edge in $T(\D)$. Otherwise the same argument for the subdiagram $\D_2$, bounded from above by $\topp(\D^-)=\bott(\D^+)$ and from below by the lower edge from $\ell$ to $r$, would show that the diagram $\D_2$ is the inverse of $\D_1$, by contradiction to the diagram $\D$ being reduced.

Since the length of the top path from $\ini$ to $r$ is $\ell+1$, the length of the bottom path from $\ini$ to $r$ is equal to $\ell+1$ modulo $n-1$.
From the minimality of $r$, there exists some $\ell'<r$ such that the bottom path from $\ini$ to $r$ is composed of the lower edges $(0,1),(1,2)\dots,(\ell'-1,\ell')$ and $(\ell',r)$. Therefore, its length $\ell'+1\equiv \ell+1 (\mod n-1)$. From the minimality of $\ell$ and the absence of a lower edge from $\ell$ to $r$, it follows that $\ell'>\ell$.
Therefore, $r-\ell>\ell'-\ell\equiv 0(\mod n-1)$ implies that $r-l>n-1$.
\end{proof}

Now we can finish the proof of Lemma \ref{lm2n}.

Consider the subdiagram $\D_1$ of the diagram $\D^+$, bounded from above by the upper edge $(\ell,r)$ and from below by the bottom path $\mathbf{bot}(\D^+)$.
The diagram $\D_1$ has at least $n-1$ inner vertices, the first $n-2$ of which are connected by arcs (i.e., edges which do not lie on $\mathbf{bot}(\D^+)$) to the terminal vertex $r$.

Let $x_i$ be the leading term of the positive part of the normal form of $\D$ (since $(\ell,r)$ is an upper edge, such $x_i$ exists).
Let $x_i^{m_0}x_{i+1}^{m_1}\cdots x_{i+p}^{m_p}$ be the longest sequence of positive powers of consecutive letters $x_j$ which forms a prefix of the normal form of $\D$.

Each positive letter $x_j$ in the normal form of $\D$ has a corresponding edge in $\D^+$. Namely, the top edge of the cell ``labeled by $x_j$'' in the algorithm described in Lemma \ref{lm1}.
From Lemma \ref{lm1} it follows that if for all $j=0,\dots,p$ the exponent $m_j=1$, then the edge corresponding to the first letter in the sequence (i.e., to $x_i$) is in fact the edge $(\ell,r)$. In this case, the $n-2$ mentioned arcs from vertices below the edge ($\ell,r)$ to $r$ show that the normal form of $\D$ starts with $\prod_{k=0}^{n-2}{x_{i+k}}$ which belongs to $H_n$. Dividing the normal form of $\D$ from the left by $\prod_{k=0}^{n-2}{x_{i+k}}$ gives an element of $\arr{F}_{\!\!n-1}$ with a shorter normal form and we are done by induction.

Thus we can assume that there exists $d=\max\{j=0,\dots,p\mid m_d\ge 2\}$. Again, from Lemma \ref{lm1} it follows that the edge corresponding to the last letter of the power $x_d^{m_d}$ is the edge $(\ell,r)$.
The arcs below $(\ell,r)$ show that the prefix $x_i^{m_0}x_{i+1}^{m_1}\cdots x_{i+d}^{m_d}$ of the normal form of $\D$ is followed by at least $n-2$ letters forming the product $\prod_{k=1}^{n-2}{x_{i+d+k}}$.
By Lemma \ref{technic}, 
 it is possible to replace $\prod_{k=0}^{d}{x_{i+k}^{m_k}}\prod_{k=1}^{n-2}{x_{i+d+k}}$
by  $\left(\prod_{k=0}^{d-1}{x_{i+n-1+k}^{m_k}}\right)x_{i+n-1+d}^{m_d-1}$ and the resulting element would still be in $\arr{F}_{\!\!n-1}$. Since its normal form is shorter than that of $\D$ we are done by induction.
\end{proof}

\begin{lemma}\label{l:58}
For all $n\ge 2$, the Thompson subalgebra $H_n$ is generated as a group by the set $A=\{x_j\cdots x_{j+n-2}\mid j\ge 0\}$.
\end{lemma}

\begin{proof}
Let $G$ be the subgroup of $F$ generated by $A$. It suffices to prove that $G$ is closed under sums. Let $a$ and $b$ be elements of $G$. In particular, $a=y_1\cdots y_t$ and $b=z_1\cdots z_s$ where $y_1,\dots, y_t$ and $z_1,\dots, z_s$ are elements in $A^{\pm 1}$.
By Lemma \ref{lm0}, $a\oplus b=(y_1\oplus \1)\cdots (y_t\oplus \1)(\1\oplus z_1)\cdots(1\oplus z_s)$. Since for any diagram $z$, we have $(1\oplus z\iv)=(1\oplus z)\iv$ and $A$ is closed under addition of $\1$ from the left, $(\1\oplus z_m)\in G$ for all $m=1,\dots,s$.
Thus it suffices to prove that for all $j\ge 0$, the element $x_j\cdots x_{j+n-2}\oplus \1\in G$.

From Lemma \ref{lm1} it follows that $$\left(\prod_{k=0}^{n-2}x_{j+k}\right)\oplus \1=\prod_{k=0}^{j}x_k\prod_{k=0}^{n-2}x_{j+k}\left(\prod_{k=0}^{j+n-1}x_k\right)\iv.$$
Let $r\in\{0,\dots,n-2\}$ be the residue of $j$ modulo $n-1$. Inserting $\prod_{k=j+1}^{j+n-r-2}x_k$ and its inverse to the right side of the above equation we get
that
$$\left(\prod_{k=0}^{n-2}x_{j+k}\right)\oplus \1=\prod_{k=0}^{j}x_k  \left(\prod_{k=j+1}^{j+n-r-2}x_k\right) \left(\prod_{k=j+1}^{j+n-r-2}x_k\right)\iv \prod_{k=0}^{n-2}x_{j+k}\left(\prod_{k=0}^{j+n-1}x_k\right)\iv$$
Since in Thompson group $F$ for all $\ell>r$ we have $x_{\ell}^{-1}x_r=x_rx_{\ell+1}^{-1}$, it is possible to replace
$\left(\prod_{k=j+1}^{j+n-r-2}x_k\right)\iv \prod_{k=0}^{n-2}x_{j+k}$ by $\prod_{k=0}^{n-2}x_{j+k} \left(\prod_{k=j+1}^{j+n-r-2}x_{k+n-1}\right)\iv$. Thus,
$$\left(\prod_{k=0}^{n-2}x_{j+k}\right)\oplus \1=\prod_{k=0}^{j}x_k  \prod_{k=j+1}^{j+n-r-2}x_k  \prod_{k=0}^{n-2}x_{j+k} \left(\prod_{k=j+1}^{j+n-r-2}x_{k+n-1}\right)\iv    \left(\prod_{k=0}^{j+n-1}x_k\right)\iv$$
Merging adjacent products, we get that
$$\left(\prod_{k=0}^{n-2}x_{j+k}\right)\oplus \1=\prod_{k=0}^{j+n-r-2}x_k \prod_{k=0}^{n-2}x_{j+k}    \left(\prod_{k=0}^{j+2n-r-3}x_k\right)\iv.$$
Since the length of each of the products $\prod_{k=0}^{j+n-r-2}x_k$, $\prod_{k=0}^{n-2}x_{j+k}$ and $\prod_{k=0}^{j+2n-r-3}x_k$ is divisible by $n-1$, it is possible to present each of them as a product of elements of $A$. The result clearly follows.


\end{proof}

\begin{cor}\label{c:59}
For all $n\ge 2$, the Thompson subalgebra $H_n$ is generated as a group by the set $\{x_j\dots x_{j+n-2}\mid j=0,\dots,n-1\}$.
\end{cor}

\begin{proof}
For every $j\ge n$, the element $x_j\dots x_{j+n-2}$ is equal to $(x_{j-n+1}\cdots x_{j-1})^{x_0\cdots x_{n-2}}$.
\end{proof}

\begin{lemma}\label{lm4n}
For all $n\ge 2$, the subgroup $\overrightarrow{F}_{\!\!n-1}$ is isomorphic to $F_n$.
\end{lemma}

\proof  The elements $x_j\cdots x_{j+n-2}$, for $j=0,\dots,n-1$ satisfy the defining
relations of $F_n$ (see \cite[page 54]{GS}). Since all proper homomorphic images of $F_n$ are Abelian \cite[Theorem 4.13]{B} and  $x_j\cdots x_{j+n-2}$, for $j=0,1$ do not commute, we get the result.
\endproof


Finally, the proofs of the following theorems are almost identical to the proofs of Theorem \ref{thm:2} and
Corollary \ref{thm:3} so we leave them to the reader.

\begin{theorem}
Let $n\ge 1$. For each $i=0,\dots, n-1$, let $S_i$ be the set of all finite binary fractions from the unit interval $[0,1]$ with sums of digits equal to $i$ modulo $n$. Then, the subgroup $\overrightarrow{F_n}$ is the intersection of the stabilizers of $S_i$, $i=0,\dots,n-1$ under the natural action of $F$ on $[0,1]$.
\end{theorem}

\begin{theorem}
For all $n\ge 1$, the subgroup $\overrightarrow{F}_{\!\!n}$ coincides with its commensurator in $F$, hence the linearization of the permutational representation of $F$ on $F/\overrightarrow{F}_{\!\!n}$ is irreducible.
\end{theorem}


\subsection{The embeddings of $F_n$ into $F$ are natural}

Suppose that $G_i=\DG(\P_i,u_i)$, $i=1,2$, are diagram groups, and we can tessellate every cell from $\P_1$ by cells from $\P_2$, which turns every cell from $\P_1$ into a diagram over $\P_2$.  Then we can define a map from $G_1$ to $G_2$ which sends every diagram $\Delta$ from $G_1$ to a diagram obtained by replacing every cell from $\Delta$ by the corresponding diagram over $\P_2$. This map is obviously a homomorphism. Such homomorphisms were called {\em natural} in \cite{GSdc}.

Instead of giving a more precise general definition of a natural homomorphism from one diagram group into another, we give an example.
 Recall that $F_n=\DG(\la x\mid x=x^n\ra,x)$. Let $\pi_n$ be the cell $x\to x^n$. It has $n+1$ vertices $0,\ldots, n$. Suppose that $n\ge3$. Connect each vertex $1,\ldots, n-2$ with the vertex $n$ by an arc. The result is an $(x,x^n)$-diagram $\Delta(\pi_n)$ over the presentation $\la x\mid x=x^2\ra$. Now for every $(x,x)$-diagram $\Delta$ from $F_n$, consider the diagram $\phi(\Delta)$ obtained by replacing every cell $\pi_n$ (resp. $\pi_n\iv$) by a copy of $\Delta(\pi_n)$ (resp. $\Delta(\pi_n)\iv$). The resulting diagram belongs to $F=\DG(\la x\mid x=x^2\ra,x)$. It is easy to check that the map $\phi$ is a homomorphism.

 A set of generators of $F_n$ is described in \cite[Page 54]{GS}. Their left-right duals also generate $F_n$. If we apply $\phi$ to these generators, we get (using Lemma \ref{lm1}) elements $x_j\cdots x_{j+n-2}$, $j=0,\ldots, n-1$, which are generators of $\arr{F}_{\!\!n-1}$ by Corollary \ref{c:59}. Since $F_n$ does not have a non-injective homomorphism with a non-Abelian image, $\phi$ is an isomorphism between $F_n$ and $\arr{F}_{\!\!n-1}$. Figure \ref{f:89} below shows the left-right duals of generators of $F_3$ from \cite{GS}, and Figure \ref{f:90} shows the images of these generators under $\phi$.


\begin{figure}[ht!]
\center{
\unitlength .5mm 
\linethickness{0.4pt}
\ifx\plotpoint\undefined\newsavebox{\plotpoint}\fi 
\begin{picture}(256,94.125)(0,0)
\put(5.5,42){\line(1,0){12.25}}
\put(85.25,42){\line(1,0){12.25}}
\put(165,42){\line(1,0){12.25}}
\put(242.75,42){\line(1,0){12.25}}
\put(229.75,42){\line(1,0){12.25}}
\put(18.5,42){\line(1,0){12.25}}
\put(98.25,42){\line(1,0){12.25}}
\put(178,42){\line(1,0){12.25}}
\put(31.5,42){\line(1,0){12.25}}
\put(111.25,42){\line(1,0){12.25}}
\put(191,42){\line(1,0){12.25}}
\put(44.5,42){\line(1,0){12.25}}
\put(124.25,42){\line(1,0){12.25}}
\put(204,42){\line(1,0){12.25}}
\put(57.5,42){\line(1,0){12.25}}
\put(137.25,42){\line(1,0){12.25}}
\put(217,42){\line(1,0){12.25}}
\put(5.75,41.75){\circle*{2}}
\put(85.5,41.75){\circle*{2}}
\put(165.25,41.75){\circle*{2}}
\put(18.25,41.75){\circle*{2}}
\put(98,41.75){\circle*{2}}
\put(177.75,41.75){\circle*{2}}
\put(255.5,41.75){\circle*{2}}
\put(242.5,41.75){\circle*{2}}
\put(31.25,41.75){\circle*{2}}
\put(111,41.75){\circle*{2}}
\put(190.75,41.75){\circle*{2}}
\put(44.25,41.75){\circle*{2}}
\put(124,41.75){\circle*{2}}
\put(203.75,41.75){\circle*{2}}
\put(57.25,41.75){\circle*{2}}
\put(137,41.75){\circle*{2}}
\put(216.75,41.75){\circle*{2}}
\put(70.25,41.75){\circle*{2}}
\put(150,41.75){\circle*{2}}
\put(229.75,41.75){\circle*{2}}
\qbezier(5.5,41.75)(38.625,87.25)(70.25,41.75)
\qbezier(85.25,41.75)(118.375,87.25)(150,41.75)
\qbezier(5.5,42)(27.75,65.125)(44,41.75)
\qbezier(98,42)(120.25,65.125)(136.5,41.75)
\qbezier(5.5,42.25)(41.125,-12.625)(70.25,42)
\qbezier(85.5,42.5)(121.125,-12.375)(150.25,42.25)
\qbezier(31,42)(50.625,15.875)(69.75,42.25)
\qbezier(111,42.25)(130.625,16.125)(149.75,42.5)
\qbezier(165,41.75)(209.375,122.125)(255.25,42)
\qbezier(190.75,41.75)(211.25,68.625)(229.75,42)
\qbezier(190.5,42)(218.5,88.375)(254.5,42.25)
\qbezier(165,41.75)(212.25,-36.375)(255.5,42)
\qbezier(190.25,42.25)(222.875,-7)(255,41.75)
\qbezier(216.5,42)(233,15.875)(254.5,42.25)
\end{picture}
}
\caption{Generators of $F_3$.}
\label{f:89}
\end{figure}
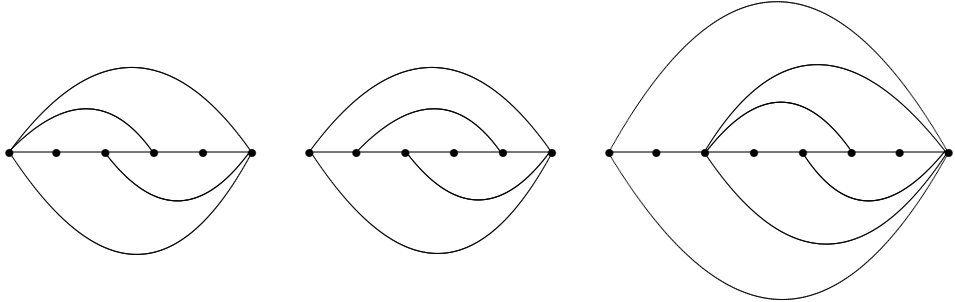

\begin{figure}[ht!]
\center{
\unitlength .5mm 
\linethickness{0.4pt}
\ifx\plotpoint\undefined\newsavebox{\plotpoint}\fi 
\begin{picture}(256,94.125)(0,0)

\put(5.5,42){\line(1,0){12.25}}
\put(85.25,42){\line(1,0){12.25}}
\put(165,42){\line(1,0){12.25}}
\put(242.75,42){\line(1,0){12.25}}
\put(229.75,42){\line(1,0){12.25}}
\put(18.5,42){\line(1,0){12.25}}
\put(98.25,42){\line(1,0){12.25}}
\put(178,42){\line(1,0){12.25}}
\put(31.5,42){\line(1,0){12.25}}
\put(111.25,42){\line(1,0){12.25}}
\put(191,42){\line(1,0){12.25}}
\put(44.5,42){\line(1,0){12.25}}
\put(124.25,42){\line(1,0){12.25}}
\put(204,42){\line(1,0){12.25}}
\put(57.5,42){\line(1,0){12.25}}
\put(137.25,42){\line(1,0){12.25}}
\put(217,42){\line(1,0){12.25}}
\put(5.75,41.75){\circle*{2}}
\put(85.5,41.75){\circle*{2}}
\put(165.25,41.75){\circle*{2}}
\put(18.25,41.75){\circle*{2}}
\put(98,41.75){\circle*{2}}
\put(177.75,41.75){\circle*{2}}
\put(255.5,41.75){\circle*{2}}
\put(242.5,41.75){\circle*{2}}
\put(31.25,41.75){\circle*{2}}
\put(111,41.75){\circle*{2}}
\put(190.75,41.75){\circle*{2}}
\put(44.25,41.75){\circle*{2}}
\put(124,41.75){\circle*{2}}
\put(203.75,41.75){\circle*{2}}
\put(57.25,41.75){\circle*{2}}
\put(137,41.75){\circle*{2}}
\put(216.75,41.75){\circle*{2}}
\put(70.25,41.75){\circle*{2}}
\put(150,41.75){\circle*{2}}
\put(229.75,41.75){\circle*{2}}
\qbezier(5.5,41.75)(38.625,87.25)(70.25,41.75)
\qbezier(85.25,41.75)(118.375,87.25)(150,41.75)
\qbezier(5.5,42)(27.75,65.125)(44,41.75)
\qbezier(98,42)(120.25,65.125)(136.5,41.75)
\qbezier(5.5,42.25)(41.125,-12.625)(70.25,42)
\qbezier(85.5,42.5)(121.125,-12.375)(150.25,42.25)
\qbezier(31,42)(50.625,15.875)(69.75,42.25)
\qbezier(111,42.25)(130.625,16.125)(149.75,42.5)
\qbezier(165,41.75)(209.375,122.125)(255.25,42)
\qbezier(190.75,41.75)(211.25,68.625)(229.75,42)
\qbezier(190.5,42)(218.5,88.375)(254.5,42.25)
\qbezier(165,41.75)(212.25,-36.375)(255.5,42)
\qbezier(190.25,42.25)(222.875,-7)(255,41.75)
\qbezier(216.5,42)(233,15.875)(254.5,42.25)
\color{red}

\qbezier(43.75,41.75)(55.125,60.125)(70,42)
\qbezier(18.25,42)(30.5,57.625)(43.75,41.75)
\qbezier(44,41.75)(56.25,26.875)(69.5,42.5)
\qbezier(18,42)(48,4)(70,42)
\qbezier(111,42)(123.375,56)(136.25,42)
\qbezier(98,41.75)(119,75.875)(149,42.5)
\qbezier(124,42)(137.625,28.5)(149.75,42)
\qbezier(98.25,42.25)(132.625,7.875)(149.5,42)
\qbezier(177.75,42.25)(215.25,105.25)(254.75,42.25)
\qbezier(229.5,42)(240,60.5)(254.5,42)
\qbezier(204,42.5)(216.75,60.25)(229.5,42)
\qbezier(229.75,42)(242.375,27.25)(254.5,42.5)
\qbezier(203.75,42)(232.625,6.25)(255,42.5)
\qbezier(177.5,42)(223.5,-18.625)(254.5,42.25)
\end{picture}
}
\caption{Images of generators of $F_3$ under $\phi$. Red arcs are added by $\phi$.}
\label{f:90}
\end{figure}
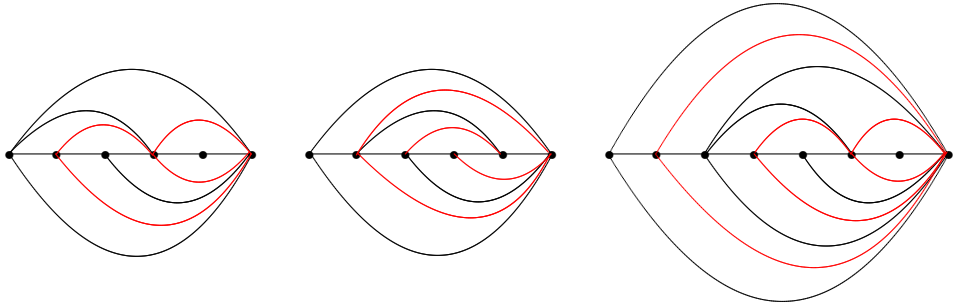

Note that $\phi$ is also an example of a \emph{caret replacement} homomorphism from $F_p$ to $F_q$ studied in \cite{BCS}. It is proved in \cite{BCS} that the image of $F_p$ under $\phi$ is undistorted in $F$.

\section{The Jones' construction}\label{s:6}

In this section we analyze the connection between elements of Thompson group $F$ and links, as explained by Jones \cite{Jo}. Given a link $L$, our analysis yields a linear bound on the word length of an element of Thompson group $F$, which represents $L$, in terms of the number of crossings and the number of unlinked unknots in $L$. Our construction essentially differs from that in \cite{Jo} only in one step (step number 2).


We will sometimes abuse notation and refer to a link and a link diagram as the same object. All link diagrams we consider are either connected (i.e., their underlying graph is connected) or with connected components "far apart", that is, with no connected component bounding another. The same is true for all plane graphs considered below. 

\subsection{Step 1. From links to signed plane graphs}

This is basically standard knot theory.
Let $L$ be a link represented by some link diagram, also denoted by $L$. 
If $L$ is connected, then the underlying graph of $L$ is a $4$-regular plane graph and in particular an Eulerian graph. As such, its dual graph is bipartite. Therefore, it is possible to color the regions of $L$ in gray and white so that the unbounded region of $L$ is white and no two adjacent regions (regions which share a boundary component) have the same color. Clearly, this is true for disconnected links as well. We define a signed plane graph $\G(L)$ as follows. We put a vertex inside every gray region and draw one edge between two vertices for every common point  of the boundaries of these regions (i.e., a crossing in the link diagram). We label an edge with "+" or "-" according to the type of the crossing as defined in Figure \ref{fig:signs}. The process is demonstrated in Figure \ref{fig:shaded_and_graph} for the link $L$ from Figure  \ref{fig:link1} .

\begin{figure}[h]
\centering
\includegraphics[width=0.2\columnwidth]{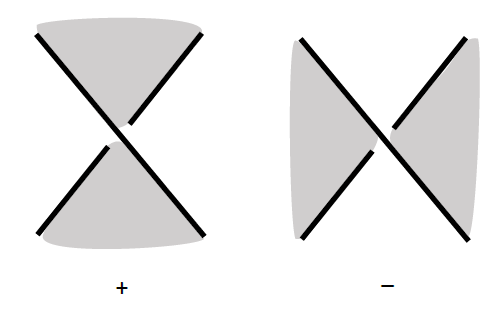}
\caption{The types of crossings.}
    \label{fig:signs}
\end{figure}

\begin{figure}[h]
\centering
\includegraphics[width=.45\columnwidth]{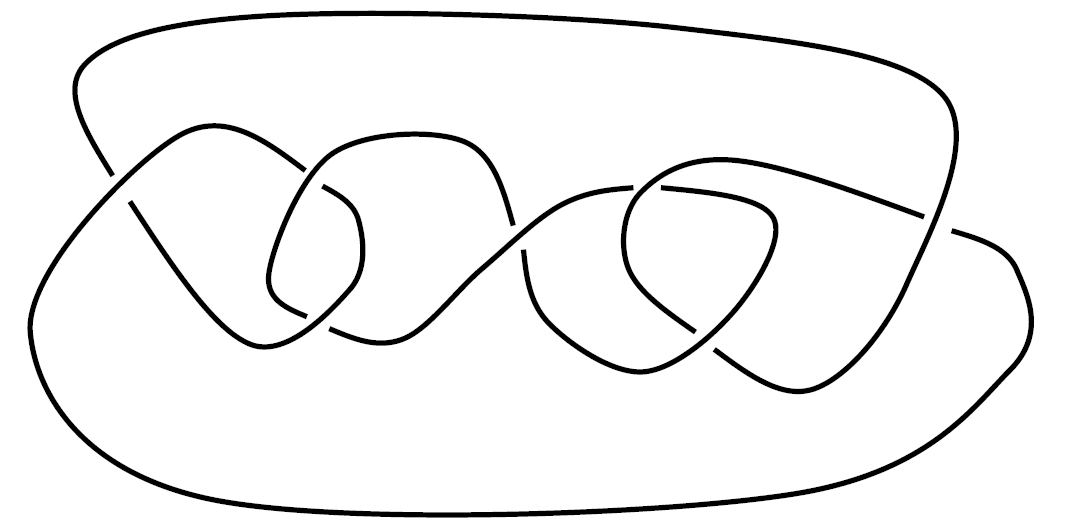}
\caption{A link L}
    \label{fig:link1}
\end{figure}

\begin{figure}[h]
\centering
\begin{subfigure}{.5\textwidth}
  \centering
  \includegraphics[width=.6\linewidth]{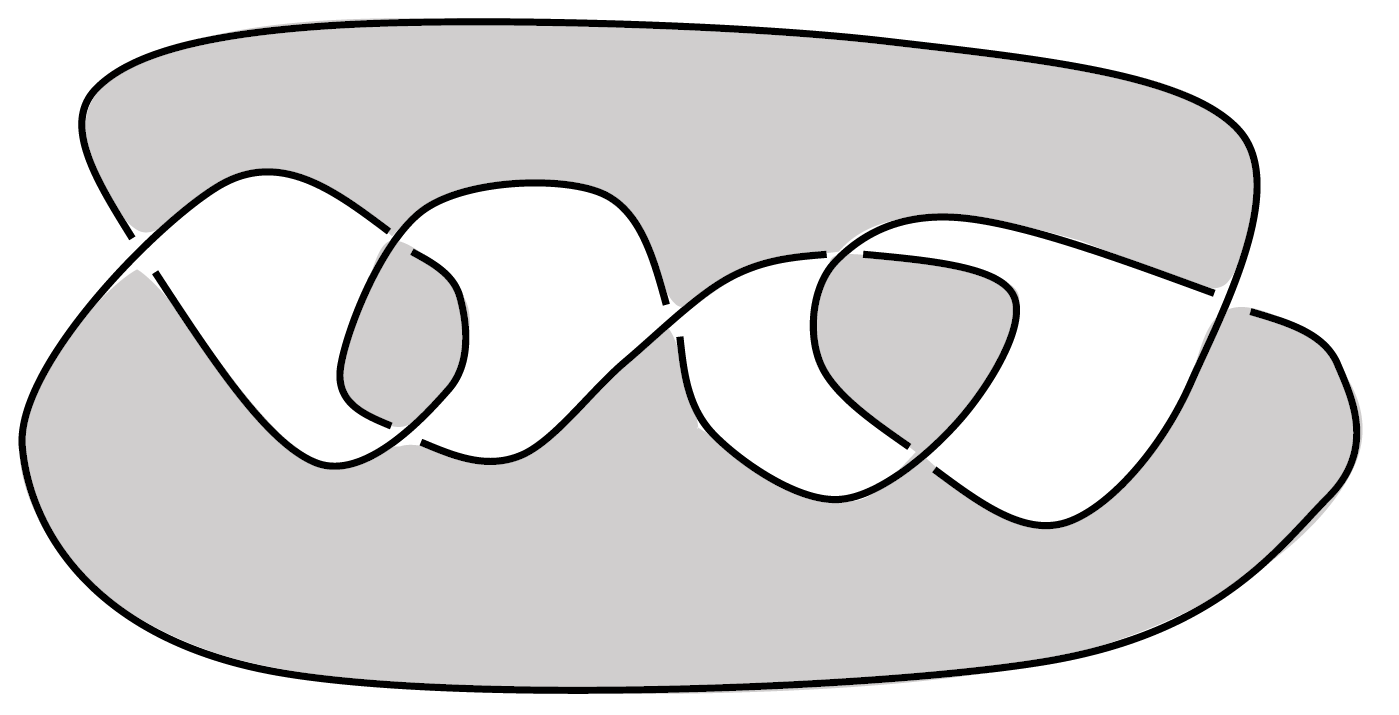}
  \caption{}
  \label{fig:}
\end{subfigure}%
\begin{subfigure}{.5\textwidth}
  \centering
  \includegraphics[width=.7\linewidth]{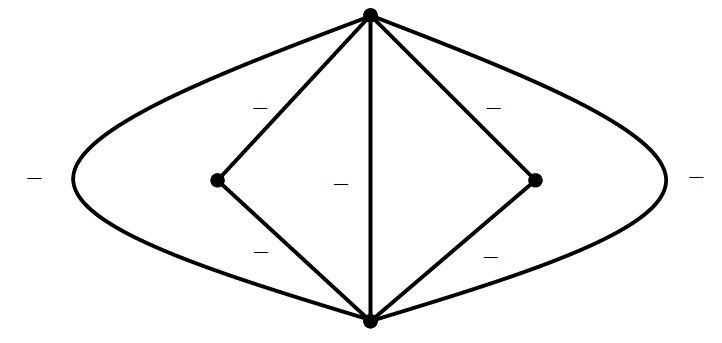}
  \caption{}
  \label{fig:shaded and graph}
\end{subfigure}
\caption{(a) Gray and white regions. (b) The signed plane graph $\G(L)$.}
\label{fig:shaded_and_graph}
\end{figure}

Clearly this step is reversible: from any plane graph $\Gamma$ with edges signed by "+" and "-" one can reconstruct a link diagram $L$ such that $\G(L)=\Gamma$. We denote this link diagram $L$ by $\La(\Gamma)$.

\subsection{Local moves on signed plane graphs}

Our goal is to turn a signed plane graph into the Thompson graph of some element from $F$. For this purpose, Jones introduced three moves.

{\bf Move of type 1:} If $v$ is a vertex of degree one, we eliminate $v$ together with the unique edge attached to it.

\begin{figure}[h!]
\centering
\includegraphics[width=0.5\columnwidth]{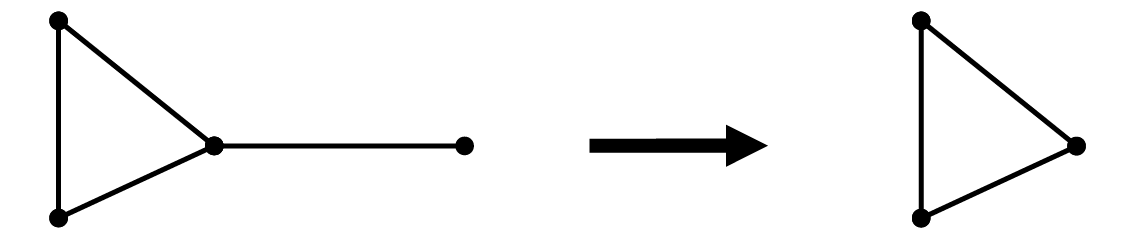}
\caption{Type 1 move}
    \label{fig:type1}
\end{figure}

{\bf Move of type 2:} Let $v$ be a vertex of degree $2$ such that the edges attached to it have opposite signs and do not share their other end vertex. We eliminate $v$ together with the edges attached to it and identify the endpoints of said edges.

\begin{figure}[h!]
\centering
\includegraphics[width=0.5\columnwidth]{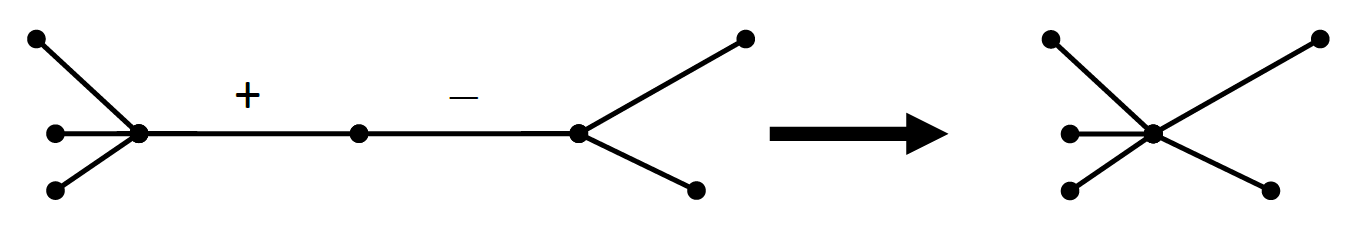}
\caption{Type 2 move}
    \label{fig:type2}
\end{figure}

{\bf Move of Type 3:} If two edges with
opposite signs $e$ and $e'$ join two vertices such that there are no other vertices and no edges between $e$ and $e'$, we erase the edges $e$ and $e'$.

\begin{figure}[h!]
\centering
\includegraphics[width=0.5\columnwidth]{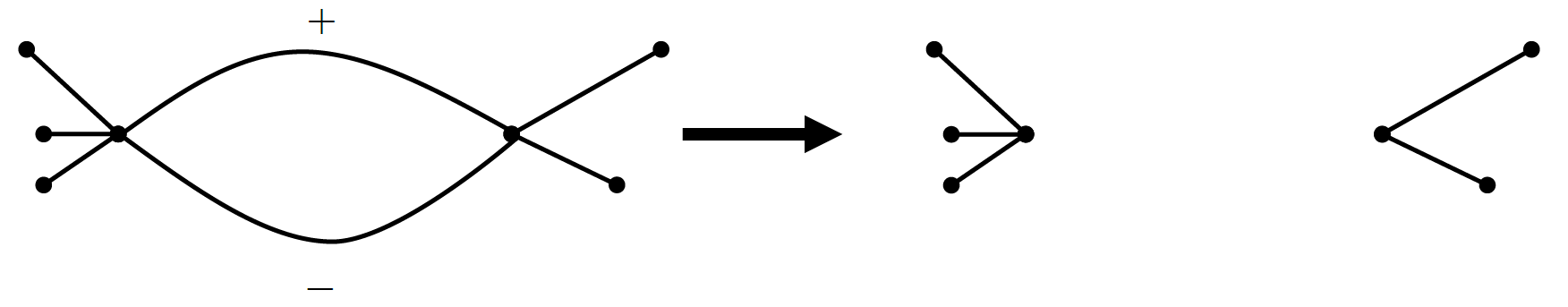}
\caption{Type 3 move}
    \label{fig:type3}
\end{figure}

If after the application of a move of type $3$ the resulting graph has a connected component which bounds another, we ``separate'' the components so that none of them would bound the other and consider it to be a part of the move.
From now on, we shall refer to a move inverse to a move of type $i$ as a move of type $i$ as well.
We say that two signed plane graphs $\Gamma_1$ and $\Gamma_2$ are \emph{equivalent} if it is possible to get from one to the other by a finite series of applications of moves of types 1-3 and isotopy of the plane.
If $L$ is a link and $\Gamma=\G(L)$ is the corresponding signed plane graph, then applying a move of type 1-3 to $\Gamma$ is equivalent to applying a Reidemeister move to $L$ (and possibly distancing some unlinked components of the link). 
Thus we have the following.

\begin{Lemma}
Let $L_1$ and $L_2$ be two link diagrams. Let $\G(L_1)$ and $\G(L_2)$ be the corresponding signed plane graphs.
If $\G(L_1)$ and $\G(L_2)$ are equivalent, then $L_1$ and $L_2$ are link diagrams of the same link.
\end{Lemma}

\begin{Remark}\label{no_loops}
Let $L$ be a link diagram with $n$ crossings. It is easy to see that by applying Reidemeister moves if necessary, it is possible to get a link diagram $L'$ with at most $n$ crossings such that the corresponding signed plane graph $\G(L')$ has no loops. Therefore, from now on all plane graphs considered here are assumed to have no loops.
\end{Remark}

%
%
%
%

\subsection{Plane graphs embeddable into a 2-page book}


\begin{definition}\label{2def}
Let $\Gamma$ be a plane graph. A \emph{$2$-page book embedding} of $\Gamma$ is a planar isotopy which takes all the vertices of $\Gamma$ to the $x$-axis and each edge of $\Gamma$ to an edge whose interior is entirely above the $x$-axis or entirely below it.
$\Gamma$ is \emph{$2$-page book embeddable} if there exists such a planar isotopy.
\end{definition}

Our definition of a $2$-page book embedding differs slightly from the standard definition. In the standard definition, the planar isotopy can be replaced by any graph embedding. Whenever we speak of $2$-page book embeddings we mean it in the sense of Definition \ref{2def}.

For a plane graph $\Gamma$, being 2-page book embeddable is equivalent to the condition that there exists an infinite oriented simple curve $\alpha$ which passes through all the vertices of $\Gamma$, does not cross any of its edges, and whose $2$ infinite rays are in the outer face of $\Gamma$. Indeed, given such a curve, an isotopy which takes $\alpha$ to the $x$-axis (oriented from left to right) is a $2$-page book embedding of $\Gamma$.

A plane graph $\Gamma$ is \emph{external Hamiltonian} if it has a Hamiltonian cycle with at least one edge on the outer face of $\Gamma$. 
A plane graph is \emph{external subhamiltonian} if it is a subgraph of some external Hamiltonian plane graph.
It is easy to see that a plane graph $\Gamma$ is $2$-page book embeddable if and only if it is external subhamiltonian. 





The following is an extension of a theorem of Kaufmann and Wiese \cite{KW}  for simple plane graphs. A graph $\Gamma$ is said to be \emph{$k$-connected} if for every set $S$ of at most $k-1$ vertices, the graph resulting from $\Gamma$ after the removal of the vertices in $S$ and the edges incident to them, is connected. 

\begin{theorem}\label{thm:kauf}
Let $\Gamma$ be a plane graph with no loops. Let $D(\Gamma)$ be the subdivision of $\Gamma$ where each edge is divided in two. Then $D(\Gamma)$ is 2-page book embeddable.
\end{theorem}

\begin{proof}
The graph $D(\Gamma)$ is a simple plane graph. As such, it is possible to triangulate it:  to complete it to a maximal simple plane graph by adding to it a finite number of edges (see Figure \ref{kaufmann}a). Let $\Gamma'$ be the resulting graph. The boundary of every face (including the outer face) of $\Gamma'$ is composed of exactly 3 edges.

We say that a triangle $\Delta(a,b,c)$ in the graph is a \emph{separating triangle} if its removal would make the graph disconnected. It is easy to see that every separating triangle is not the boundary of a face. By a result of Whitney \cite{Wh} (see also \cite{Ch,He}), every maximal simple plane graph with no separating triangles is external Hamiltonian. Thus, if $\Gamma'$ has no separating triangles, it is $2$-page book embeddable and we are done. (Indeed, $D(\Gamma)$ embeds into $\Gamma'$). Assume that $\Gamma'$ has a separating triangle. We will need the following Lemma.

%

\begin{lemma}\label{above}
Let $a,b,c$ be the vertices of a triangle in $\Gamma'$. Then, at least one of the edges $(a,b),(b,c)$ and $(c,a)$ is not a sub-edge of any of the edges of $\Gamma$.
\end{lemma}

\begin{proof}
The vertices of $\Gamma'$ can be divided into two disjoint sets: the vertices of $\Gamma$ and the midpoints of edged of $\Gamma$.
If $(a,b)$ is a sub-edge of an edge of $\Gamma$ then without loss of generality, we can assume that $a$ is a vertex of $\Gamma$ and $b$ is a midpoint of an edge of $\Gamma$. Since $a$ and $c$ are joined by an edge in $\Gamma'$, $c$ must be a midpoint of an edge of $\Gamma$. Indeed, in $\Gamma'$, all the neighbors of a vertex of $\Gamma$ are midpoints of edges of $\Gamma$. Then, both end vertices of the edge $(b,c)$ are midpoints of edges of $\Gamma$. Therefore, $(b,c)$ is not a sub-edge of any edge of $\Gamma$.
\end{proof}

Now we can complete the proof. If $\Delta(a,b,c)$ is a separating triangle 
we choose an edge of $\Delta(a,b,c)$ which is not a sub-edge of any edge of $\Gamma$. We separate it into two edges and triangulate the resulting graph. The number of separating triangles in the graph decreases after this step, no new separating triangle occurs and Lemma \ref{above} still holds for the resulting graph.

After a finite number of iterations (for an illustration see Figure \ref{kaufmann}b), we get a maximal simple plane graph with no separating triangles and thus a $2$-page book embeddable graph.
Note that during the process of eliminating the separating triangles of $\Gamma'$, we have only divided edges which are not sub-edges of edges of $\Gamma$. Thus, the subdivision $D(\Gamma)$, where every edge of $\Gamma$ is divided in two, is a subgraph of the resulting maximal simple plane graph and in particular it is 2-page book embeddable.
\end{proof}

\begin{figure}[h!]
\captionsetup{width=0.8\columnwidth}
\centering
\includegraphics[width=.85\columnwidth]{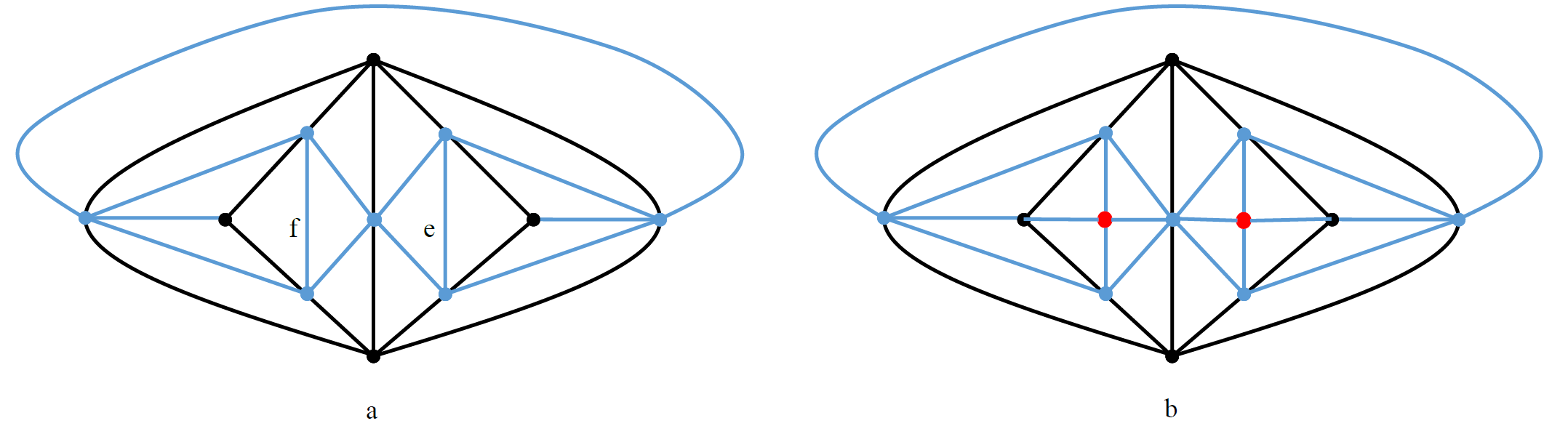}
\caption{(a) a triangulation of the subdivision $D(\Gamma)$ of $\Gamma$ from Figure \ref{fig:shaded and graph}.
(b) The graph $\Gamma'$ after two iterations. The edges $e$ and $f$ were divided in two and the graph was triangulated. There are no more separating triangles so the graph is $2$-page book embeddable.}
    \label{kaufmann}
\end{figure}

\subsection{Step 2. From signed plane graphs to 2-page book embedded graphs}


Let $\Gamma$ be a signed plane graph with no loops. At first we consider the graph $\Gamma$ as a non labeled plane graph. Let $D(\Gamma)$ be the subdivision of $\Gamma$ where every edge is divided in two. By Theorem \ref{thm:kauf}, $D(\Gamma)$ is $2$-page book embeddable. Therefore there exists an oriented simple curve $\alpha$ such that $\alpha$ passes through all the vertices of $D(\Gamma)$, does not cross any of its edges and the two infinite rays of $\alpha$ are in the outer face of $D(\Gamma)$. See Figure \ref{fig:alpha_a}.

Let $v$ be a midpoint of an edge $e$ of $\Gamma$, considered as a vertex of $D(\Gamma)$. When the curve $\alpha$ passes through $v$, it either remains on the same side of $e$ or crosses the edge $e$. In the first case, we can eliminate the vertex $v$ from $D(\Gamma)$ and adjust the curve $\alpha$ accordingly. The result would be a 2-page book embeddable graph (with one less vertex) together with a suitable curve $\alpha$. Eliminating all the possible midpoints of edges of $\Gamma$ in this way, results in a plane graph $\Gamma'$. see Figure \ref{fig:alpha_b}.


\begin{figure}[h]
\captionsetup{width=0.8\columnwidth}

\centering
\begin{subfigure}{.55\textwidth}
  \centering
  \includegraphics[width=.4\linewidth]{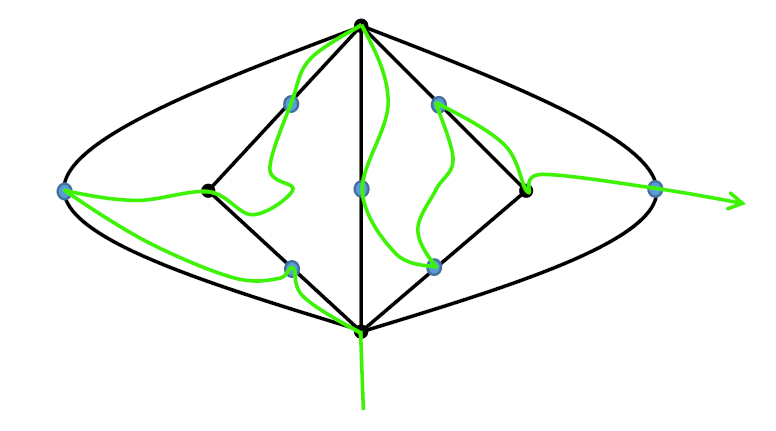}
  \caption{}
  \label{fig:alpha_a}
\end{subfigure}%
\begin{subfigure}{.5\textwidth}
  \centering
  \includegraphics[width=.6\linewidth]{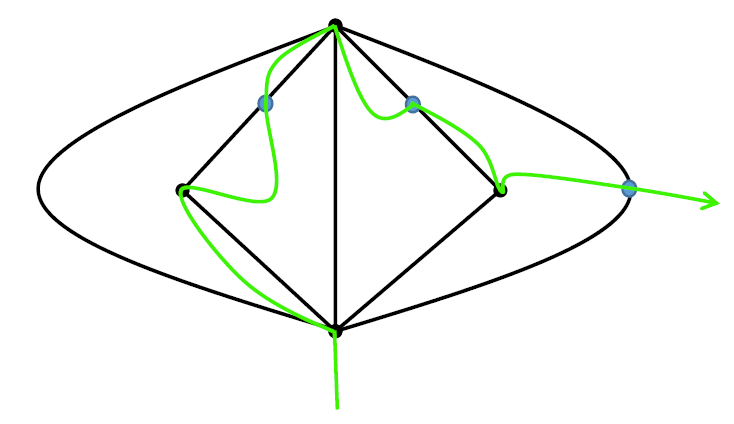}
  \caption{}
  \label{fig:alpha_b}
\end{subfigure}
\caption{(a) The subdivision $D(\Gamma)$ and the curve $\alpha$. (b) Midpoints of edges not crossed by $\alpha$ are erased and the curve $\alpha$ is adjusted.}
\label{fig:test}
\end{figure}

We shall consider the isotopy of the plane which takes the directed curve $\alpha$ to the $x$-axis directed from left to right. In accordance with this isotopy, we say that edges to the right of $\alpha$, are below it and edges to the left of $\alpha$ are above it.


Let $v$ be a vertex of $\Gamma'$ which is a midpoint of some edge $e$ of $\Gamma$. Then, up to reflections, the curve $\alpha$ passes through $v$ as in Figure \ref{fig:pass}.

 \begin{figure}[h!]
\centering
\includegraphics[width=.3\columnwidth]{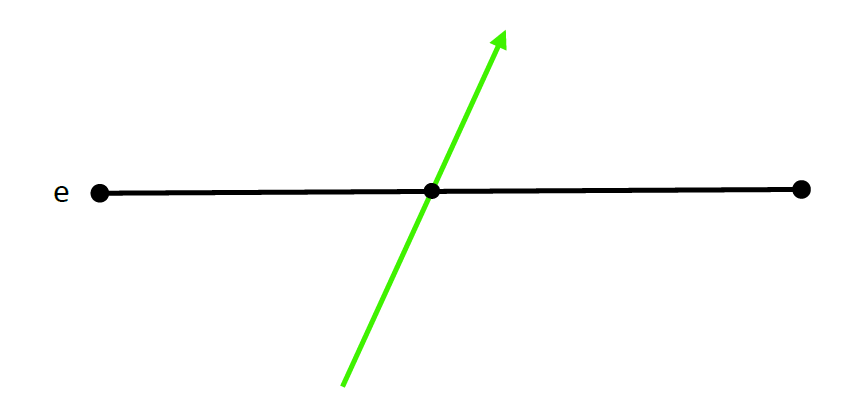}
\caption{The curve $\alpha$ crosses every edge of $\Gamma$ divided in two in $\Gamma'$.}
    \label{fig:pass}
\end{figure}

\begin{Lemma}\label{divide}
It is possible to add an additional vertex to the edge $e$ (thus, dividing it to $3$ sub-edges) and adjust the curve $\alpha$ accordingly, so as to remain with a $2$-page book embeddable graph and a suitable curve $\alpha$. Moreover, it can be done, so that $2$ of the $3$ sub-edges of $e$ would be above (below) $\alpha$ and the other one below (above) it.
\end{Lemma}

\begin{proof}
The process is demonstrated in Figure \ref{adjust}.

 \begin{figure}[h!]
 \captionsetup{width=0.8\columnwidth}

\centering
\includegraphics[width=.65\columnwidth]{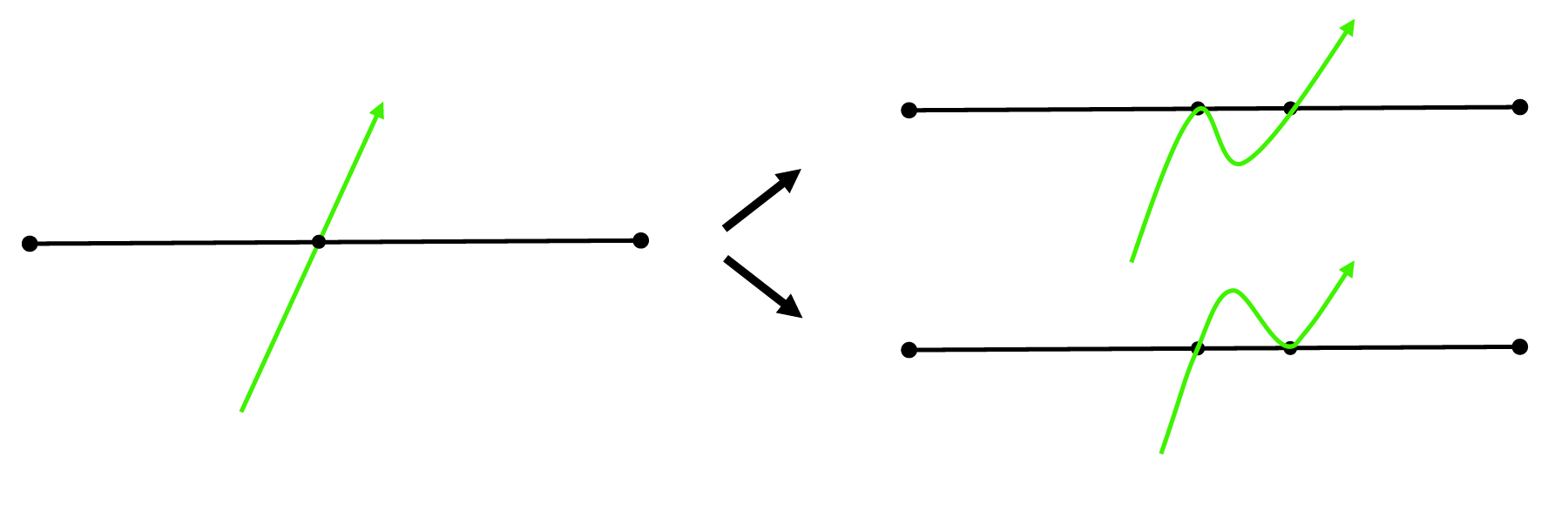}
\caption{Adding a vertex to an edge of $\Gamma$ cut in two in $\Gamma'$ and adjusting the curve $\alpha$ accordingly. In the right top (bottom) figure, two of the $3$ sub-edges of $e$ are above (below) $\alpha$ and the other one is below (above) it.}
    \label{adjust}
\end{figure}
\end{proof}

Now we consider the labels of the edges of $\Gamma$. For each edge of $\Gamma$ which is not divided in two in the transition to $\Gamma'$, we assign the corresponding edge of $\Gamma'$ its label in $\Gamma$.
If an edge $e$ of $\Gamma$ is replaced by two edges in $\Gamma'$ we do the following. If $e$ is labeled by "+" ("-") in $\Gamma$ we use Lemma \ref{divide} to replace its $2$ sub-edges in $\Gamma'$ by $3$ sub-edges and adjust the curve $\alpha$ so that two of the $3$ sub-edges are above (below) $\alpha$ and the other one is below (above) it. We label the $3$ resulting edges of $\Gamma'$ according to their position with respect to $\alpha$, where an edge above $\alpha$ is labeled by "+" and an edge below $\alpha$ is labeled by "-". We call the resulting graph $\Gamma''$. See Figure \ref{fig:3subdivision1}.

 \begin{figure}[h!]
\centering
\includegraphics[width=.5\columnwidth]{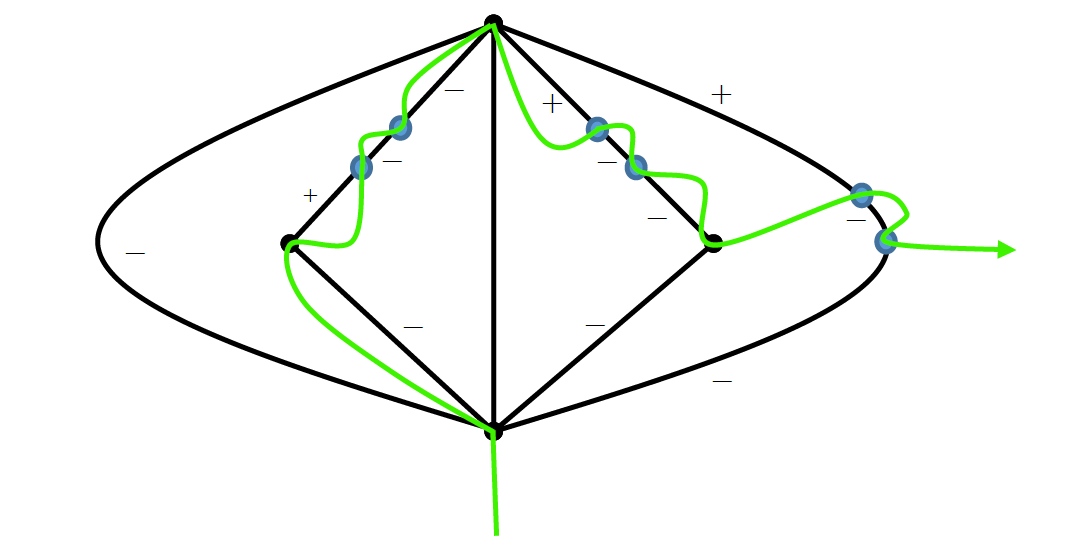}
\caption{The graph $\Gamma''$ with the curve $\alpha$.}
    \label{fig:3subdivision1}
\end{figure}

\begin{proposition}
The graph $\Gamma''$ is equivalent to the graph $\Gamma$.
\end{proposition}

\begin{proof}
The graph $\Gamma''$ results from the graph $\Gamma$ by dividing some of its edges to $3$ sub-edges and assigning them labels. If $e$ is an edge of $\Gamma$ labeled by "+" in $\Gamma$ and replaced by $3$ edge in $\Gamma''$, then two of the edges replacing it are labeled by "+". Therefore, if $e_+$ and $e_-$ are the end vertices of $e$ in $\Gamma$, then either the sub-edge of $e$ incident to $e_-$ or the sub-edge of $e$ incident to $e_+$ is labeled by "+". Assume the sub-edge of $e$ incident to $e_-$ is labeled by "+". Then, applying a type 2 move on the graph $\Gamma''$, it is possible to eliminate the other two sub-edges of $e$ (which meet at a vertex of degree $2$ and have opposite signs) and identify their end vertices. We get the original edge $e$ labeled by "+". 
\end{proof}

Thus we have proved the following.

\begin{Theorem}\label{thm:2page}
Let $\Gamma$ be a signed plane graph with no loops. Then, there exists a signed plane graph $\Gamma''$ such that
\begin{enumerate}
\item $\Gamma''$ is equivalent to $\Gamma$.
\item $\Gamma''$ is a subdivision of $\Gamma$ where each edge of $\Gamma$ is divided to $3$ sub-edges or not divided at all.
\item $\Gamma''$ can be $2$-page book embedded, so that for every edge $e$ of $\Gamma$ which is replaced by $3$ sub-edges in $\Gamma''$, the signs of all $3$ sub-edges in $\Gamma''$ are compatible with their embedding below or above the $x$-axis.
\end{enumerate}
\end{Theorem}

%

\subsection{Step 3: From $2$-page book embeddable graphs to Thompson graphs}

Let $\Gamma$ be a signed plane $2$-page book embeddable graph with no loops. We use an isotopy which takes all the vertices of $\Gamma$ to the $x$ axis and each edge of $\Gamma$ to an edge above or below the $x$-axis. Following \cite{Jo} we call the resulting graph \emph{standard}.

The image of $\Gamma''$ from Figure \ref{fig:3subdivision1} under the isotopy which takes $\alpha$ to the $x$-axis (directed from left to right) appears in Figure \ref{2page}. All the unlabeled edges in the figure have labels compatible with their position above or below the $x$-axis. The graph in the figure is standard.

Given a standard graph we enumerate its vertices from left to right by $0,1,\dots$. The \emph{inner vertices} of a standard graph are all the vertices except for vertex number $0$.



 \begin{figure}[h!]
\centering
\includegraphics[width=.65\columnwidth]{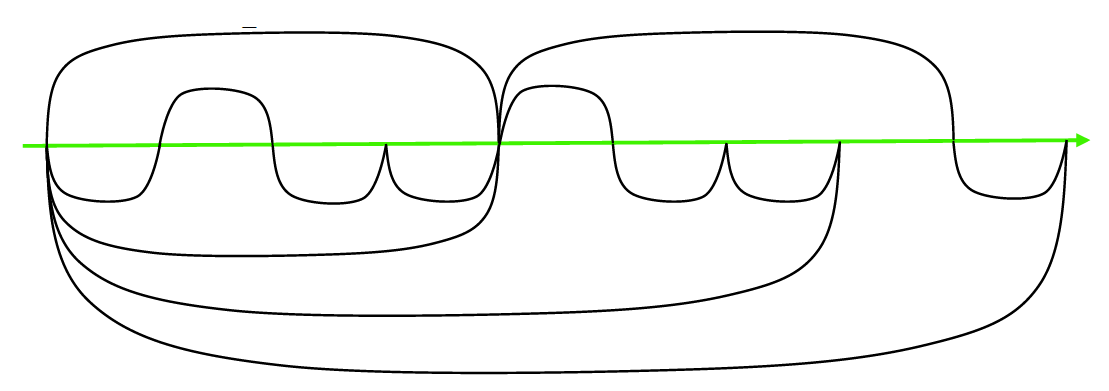}
\caption{A $2$-page book embedding of the graph $\Gamma''$.}
    \label{2page}
\end{figure}

\begin{definition}
Let $\Gamma$ be a standard graph. We orient all the edges of $\Gamma$ from left to right. The graph  $\Gamma$ is called \emph{Thompson}
 if every inner vertex of $\Gamma$ has exactly one upper incoming edge (i.e., an edge above the $x$-axis) and one lower incoming edge (i.e., an edge below the $x$-axis) and in addition all the upper edges in $\Gamma$ are labeled by "+" and all the lower edges in $\Gamma$ are labeled by "-".
\end{definition}


\begin{Lemma}\cite[Lemma 5.3.3]{Jo}\label{process}
Let $\Gamma$ be a standard graph. Then $\Gamma$ is equivalent to a Thompson graph.
\end{Lemma}

\begin{proof}
The proof follows closely the proof of \cite[Lemma 5.3.3]{Jo}. We say that a vertex is \emph{good} if it has at least one upper and one lower incoming edge. We say that an edge is \emph{good} if its label is compatible with its position with respect to the $x$-axis and it is not \emph{superfluous}.  An upper (lower) edge is said to be superfluous if it is not the bottom-most (top-most) incoming upper (lower) edge of its right end vertex.
If $\Gamma$ is not Thompson, it has a bad vertex or a bad edge (where bad is the opposite of good). Thus, at least one of the following occurs in $\Gamma$.
\begin{enumerate}
\item There is an inner vertex in $\Gamma$ with no incoming edges.
\item There is an inner vertex in $\Gamma$ with no upper (lower) incoming edge but with a lower (upper) incoming edge.
\item There is a superfluous edge in $\Gamma$. 
\item There is an edge in $\Gamma$ with a label non-compatible with its position. 
\end{enumerate}
We turn $\Gamma$ into a Thompson graph in $4$ steps, taking care of each of the above problems in turn. All the edges and vertices added during this process are good. In all the figures in this proof (namely, Figures \ref{fig:1}-\ref{fig:4}), all the unlabeled black edges can have arbitrary signs. The non-black edges (i.e., the edges attached during this process) have signs compatible with their position. 

Step 3.1. If an inner vertex $i$ has no incoming edges, we connect the vertices $(i-1)$ and $i$ by two edges. One upper edge labeled by "+" and one lower edge labeled by "-" as done in Figure \ref{fig:1}. We do it using a type $3$ move. The vertex $i$ becomes good.

  \begin{figure}[h!]
\centering
\includegraphics[width=.7\columnwidth]{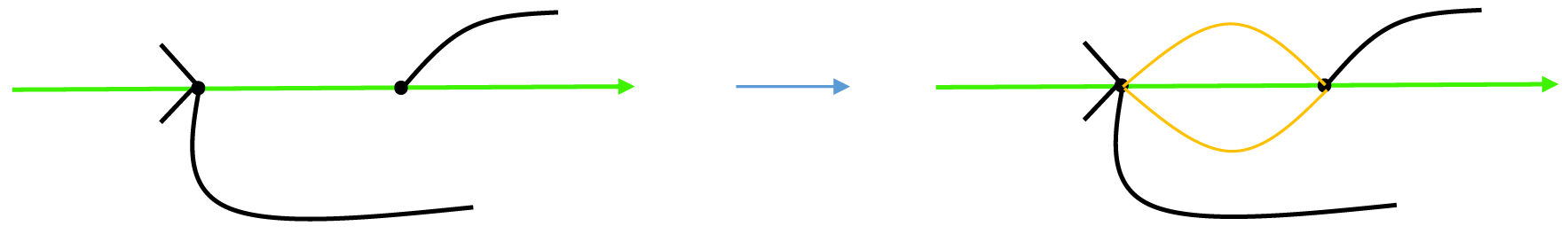}
\caption{Turning a vertex with no incoming edges to a good vertex.}
    \label{fig:1}
\end{figure}

Applying this step to all the relevant vertices, we can assume that every inner vertex has at least one incoming edge.

Step 3.2. Let $i$ be an inner vertex with no lower incoming edge (the case where $i$ has no upper incoming edge is similar). To correct the problem we apply a move of type 1 followed by a move of type 3 as demonstrated in Figure \ref{fig:2} (the vertices in the left figure are the vertices $(i-1)$ and $i$). One new vertex is added in the process.

  \begin{figure}[h!]
\centering
\includegraphics[width=.88\columnwidth]{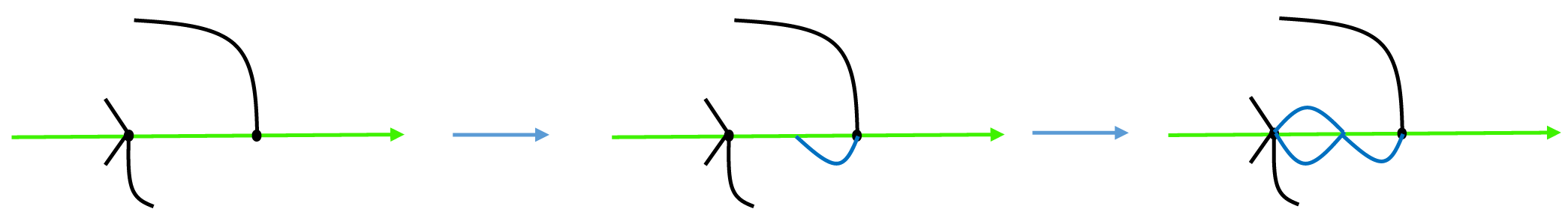}
\caption{Adding an incoming lower edge.}
    \label{fig:2}
\end{figure}

After several applications of this step, we can assume that all vertices have incoming upper and lower edges and are thus good.

Step 3.3. Let $e$ be a superfluous upper edge in $\Gamma$ (the case of a superfluous lower edge is similar). We assume that $e$ is the top-most upper incoming edge of its right end vertex $v$. We apply a move of type 2 to separate the vertex $v$ to two vertices connected by a path of two edges with opposite signs. All the incoming edges of $v$, other than $e$, remain connected to its left copy (which we consider to be the vertex $v$). The edge $e$ as well as the outgoing edges of $v$ are connected to the right copy of $v$ (which we count as a new vertex). Then we apply a move of type 1 followed by a move of type 3 (as in Step 3.2). The process is demonstrated in Figure \ref{fig:3}. All the vertices after this step are good. All the new edges are good and the edge $e$ is no longer superfluous. 
$3$ new vertices are added in this step.

  \begin{figure}[h!]
\centering
\includegraphics[width=.85\columnwidth]{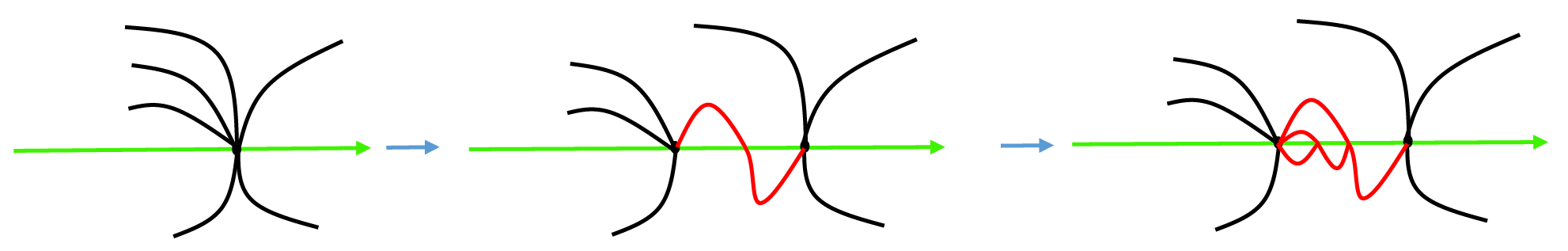}
\caption{Fixing a superfluous edge.}
    \label{fig:3}
\end{figure}

After several iterations of this step we can assume that all vertices are good and there are no superfluous edges.

Step 3.4. Let $e$ be an edge with a label incompatible with its position. We assume that $e$ is an upper edge labeld by "-", the other case being similar. We fix the problem as demonstrated in Figure \ref{fig:4}. We apply two moves of type 2 (see the second figure from the left), then apply an isotopy which turns $e$ to a lower edge, followed by several moves of type 1 and 3 as in Step 3.2. $8$ new vertices are added in this step.

  \begin{figure}[h!]
\centering
\includegraphics[width=.85\columnwidth]{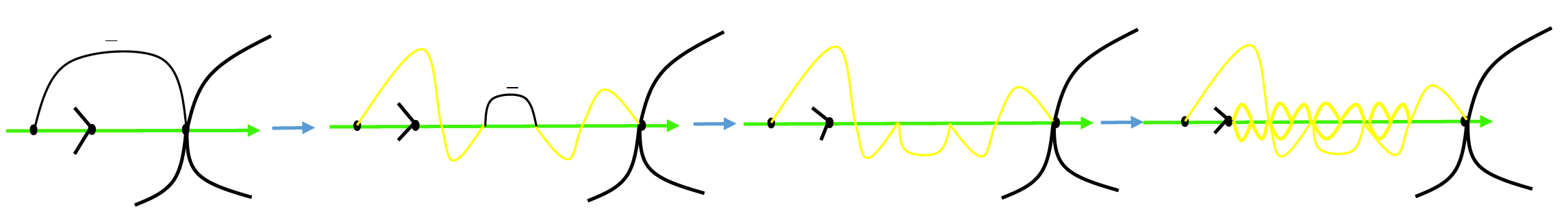}
\caption{Moving a "-" edge down.}
    \label{fig:4}
\end{figure}

After applying this step to every edge with a non compatible label, we get a graph where all the vertices and all the edges are good. This is a Thompson graph and it is clearly equivalent to the original graph $\Gamma$.
\end{proof}

\begin{Example}
We demonstrate the process of turning a standard graph to an equivalent Thompson graph on the graph $\Gamma''$ from Figure \ref{2page}. Since every vertex in the graph has at least one incoming edge, step 3.1 from the proof of Lemma \ref{process} is not necessary. Steps 3.2, 3.3 and 3.4 are depicted in Figures \ref{step2}, \ref{step3} and \ref{step4}. All the unlabeled edges in the figures have signs compatible with their positions. The graph in Figure \ref{step4} is Thompson. It is equivalent to $\Gamma''$.

 \begin{figure}[h!]
\centering
\includegraphics[width=.65\columnwidth]{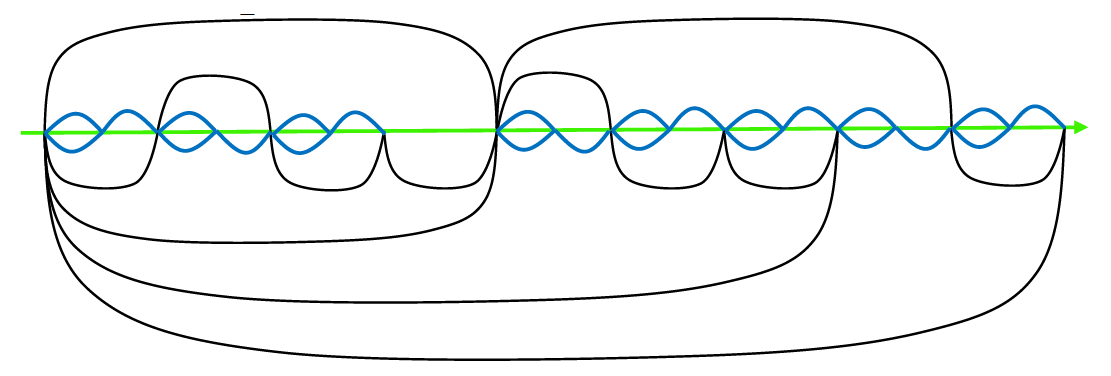}
\caption{The graph after an application of Step 3.2.}
    \label{step2}
\end{figure}

  \begin{figure}[h!]
\centering
\includegraphics[width=.7\columnwidth]{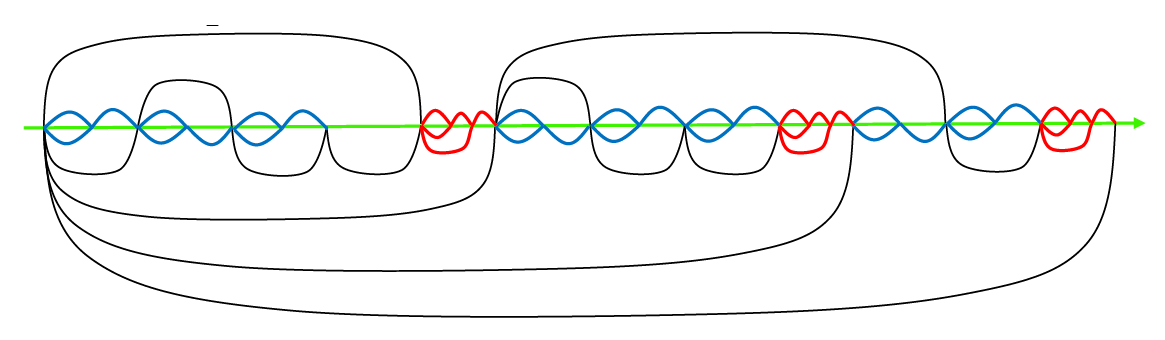}
\caption{The graph after an application of Step 3.3.}
    \label{step3}
\end{figure}

 \begin{figure}[h!]
\centering
\includegraphics[width=.8\columnwidth]{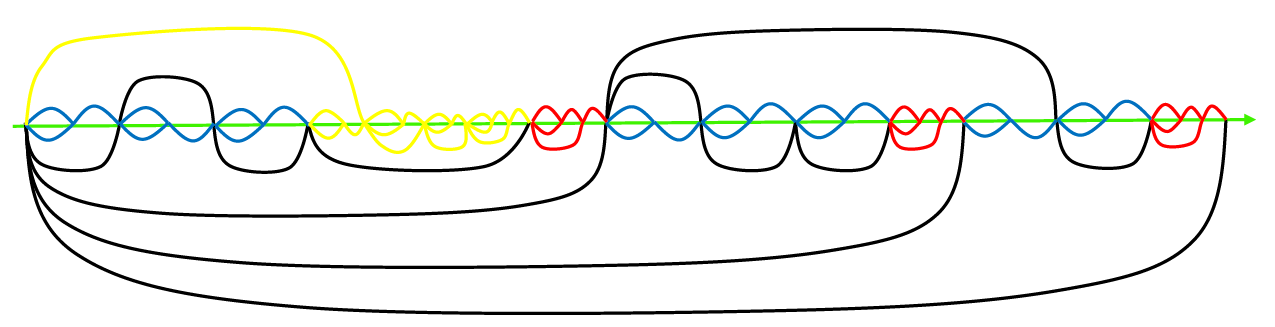}
\caption{The graph after an application of Step 3.4.}
    \label{step4}
\end{figure}

\end{Example}


\begin{Lemma}\label{l:cv}
Let $L$ be a link which contains $u$ unlinked unknots. Assume that $L$ is represented by a link diagram with $n$ crossings. Then, there exists a Thompson graph $\Gamma$ with at most $12n+u$ vertices such that $\La(\Gamma)$ is equal to $L$.
\end{Lemma}

\begin{proof}
We shall assume that the link $L$ consists of a single component which cannot be separated into several unlinked components (the general case will be addressed later below). Let $L'$ be a link diagram which represents $L$. By assumption, $L'$ is connected. Also, by Remark \ref{no_loops} we can assume that $\Gamma=\G(L')$ has no loops.
Since every edge of $\Gamma$ corresponds to a crossing in $L'$, we have $|E(\Gamma)|=n$.
If $\Gamma$ is a tree, then it is equivalent (by means of type 1 moves) to a Thompson graph composed of a single vertex. This Thompson graph clearly satisfies the conclusion of the lemma. Thus, we can assume that $\Gamma$ is not a tree. $L'$ being connected implies that $\Gamma$ is connected, therefore $|V(\Gamma)|\le |E(\Gamma)|=n$.
Let $\Gamma''$ be the $2$-page book embeddable graph equivalent to $\Gamma$, from Theorem \ref{thm:2page} and consider its embedding defined by the curve $\alpha$ from the theorem.

Let $m$ be the number of edges in $\Gamma$ which are divided to $3$ edges in $\Gamma''$.
Then $|E(\Gamma'')|=n+2m$ whereas $|V(\Gamma'')|\le n+2m$.

Now we apply the process explained in Lemma \ref{process} to the standard graph $\Gamma''$ and consider the number of additional vertices at each step. In the analysis we distinguish between the original vertices and edges of $\Gamma''$  and those added during the process.

In Step 3.1, no new vertices are added.

Assume that $x$ vertices are dealt with  in Step 3.2. For each of these vertices one new vertex is added. Thus, $x$ new vertices are created in this step. All of them are good.

In Step 3.3 we deal with superfluous edges. Every edge added during the process is good. Hence only the $n+2m$ original edges of $\Gamma''$ can be superfluous. Each of the $x$ vertices dealt with in the second step has an original edge of $\Gamma''$ as an incoming edge (otherwise it would have been dealt with in Step 3.1). If that edge is upper (lower), then all the upper (lower) incoming edges of the vertex are original edges of $\Gamma''$. The bottom-most (top-most) of them is not superfluous. Hence at least $x$ of the original edges of $\Gamma''$ are not superfluous. Thus dealing with at most $(n+2m-x)$ edges results in at most $3(n+2m-x)$ additional vertices.

Finally, in Step 3.4 we "move" edges whose labels are incompatible with their positions above or below the $x$-axis. Only the original edges of $\Gamma''$ which are edges of $\Gamma$ (as opposed to sub-edges of edges of $\Gamma$) can have incompatible labels. Moving  each of these edges requires the addition of $8$ new vertices. Thus, at most $8(n-m)$ new vertices are added in this step.

The total number of vertices in the resulting Thompson graph is at most $n+2m+x+3(n+2m-x)+8(n-m)\le 12n$, as requested.

Now, in general, let $L$ be a link composed of two unlinked components $L_1$ and $L_2$. Let $L'$ be the link diagram of $L$. We can assume that $L'$ has two connected components $L_1'$ and $L_2'$ corresponding to $L_1$ and $L_2$. Indeed, separating unlinked components of a link diagram can be done without increasing the number of crossings in the diagram. Let $\Gamma_1$ and $\Gamma_2$ be the Thompson graphs of $L_1$ and $L_2$, which satisfy the conclusion of the lemma with respect to the number of crossings and unlinked unknots in $L_1'$ and $L_2'$. Then, drawing $\Gamma_2$ to the right of $\Gamma_1$ and attaching the left-most vertex of $\Gamma_2$ to the right-most vertex of $\Gamma_1$ results in a Thompson graph $\Gamma_3$ corresponding to the link $L$. Since no vertices are added in this construction and the upper bound $12n+u$ is linear, it holds in the general case as well.
\end{proof}

\subsection{Step 4: From Thompson graphs to elements of $F$}\label{ss:tf}

Let $\Gamma$ be a standard graph which is Thompson. As an unlabeled graph, $\Gamma$ is the Thompson graph $T(\Delta)$ of some $(x,x)$-diagram $\Delta$ over the semigroup presentation $\la x\mid x=x^2\ra$.
Indeed, considering the edges of $\Gamma$ to be the left edges of $\Delta$ (in accordance with Remark \ref{left}), it is possible to complete it to a diagram $\Delta$ by attaching suitable right edges. The diagram $\Delta$ does not have to be reduced but it has no $\pi\iv\circ\pi$ dipoles.

In particular, we can finally conclude that the link $L$ from Figure \ref{fig:link1} corresponds to the following element of Thompson group $F$:

$$\begin{array}{l}(x_0x_1x_2^2x_4x_5x_6x_{18}^2x_{20}x_{21}x_{22}x_{23}x_{24}x_{25}x_{26}x_{27}) \\
\hskip .5 in ( x_0^4x_2x_3x_4^2x_6^2x_7x_8^2x_{10}^2x_{12}^2x_{15}^2x_{18}x_{19}x_{20}^2x_{22}^2x_{24}^2x_{27}x_{28}x_{29}^2x_{31}^2 )\iv.\end{array}$$

\subsection{The Thompson index of a link}

Given a diagram $\Delta$ over $\la x\mid x=x^2\ra$, we define the signed plane graph $\Gamma(\Delta)$ to be the Thompson graph $T(\Delta)$ with all the upper edges labeled by "+" and all the lower edges labeled by "-". We say that $\Delta$ \emph{represents} the link $L$ if $\G(L)=\Gamma(\Delta)$.

\begin{Lemma}\label{reduced_link}
Let $\Delta$ be an $(x,x)$- diagram over the semigroup presentation $\la x\mid x=x^2\ra$ such that $\Delta$ has no dipoles of the form $\pi\iv\circ\pi$.
\begin{enumerate}
\item If $\Delta$ is not reduced then the link it represents $L=\La(\Gamma(\Delta))$ has an unlinked unknot as a component.
\item If $\Delta'$ results from $\Delta$ by reducing a $\pi\circ\pi^{-1}$ dipole, then the link it represents $L'=\La(\Gamma(\Delta'))$ results from $L$ by the removal of one unlinked unknot.
\end{enumerate}
\label{dipole_link}
\end{Lemma}

\begin{proof}
It suffices to prove part (2). To get the graph $T(\Delta')$ we remove from $T(\D)$ a vertex with exactly two edges, one upper and one lower, connecting it to another vertex of $T(\Delta)$.
Thus $\Gamma(\Delta')$ results from $\Gamma(\D)$ by an application of a move of type $3$ and the subsequent removal of a vertex which becomes isolated.
A move of type $3$ does not affect the associated link. Since an isolated vertex in a graph corresponds to an  unknot in a link diagram, its removal implies the removal of one unlinked unknot from the corresponding link.
\end{proof}

Let $\Delta$ be a reduced nontrivial diagram in $F$ we denote by  $\call(\Delta)$ the link represented by $\Delta$.

\begin{Lemma}\label{plus_one} The link $\call(\1\oplus\Delta)$ results from the link $\call(\Delta)$ by an addition of one unlinked unknot.
\end{Lemma}

\begin{proof}
To get from the Thompson graph $T(\Delta)$ to the Thompson graph $T(\1\oplus\Delta)$ one has to add a vertex to the left of the initial vertex of $T(\Delta)$ and connect it to the initial vertex of $T(\Delta)$ by two edges; an upper edge and a lower edge. Therefore, $\Gamma(\1\oplus\Delta)$ results from $\Gamma(\Delta)$ by the addition of an isolated vertex to the left of $\Gamma(\Delta)$ (hence, an addition of an unlinked unknot to the associated link) followed by a type $3$ move which does not affect the associated link.
\end{proof}

\begin{Lemma}\label{k+2}
Let $L$ be the link which consists of $u\ge 1$ unlinked unknots. Then, $L$ is represented by the element $x_{u-1}$ of Thompson group $F$. In particular, it is represented by a reduced diagram with $u+3$ vertices.
\end{Lemma}

\begin{proof}
It is easy to check that $x_0$ represents the unknot (see Figure \ref{fig:x0}). Since for all $i\ge 0$, we have, $x_{i+1}=\1\oplus x_i$, the result follows from Lemma \ref{plus_one}.
\end{proof}

 \begin{figure}[ht!]
\centering
\includegraphics[width=.65\columnwidth]{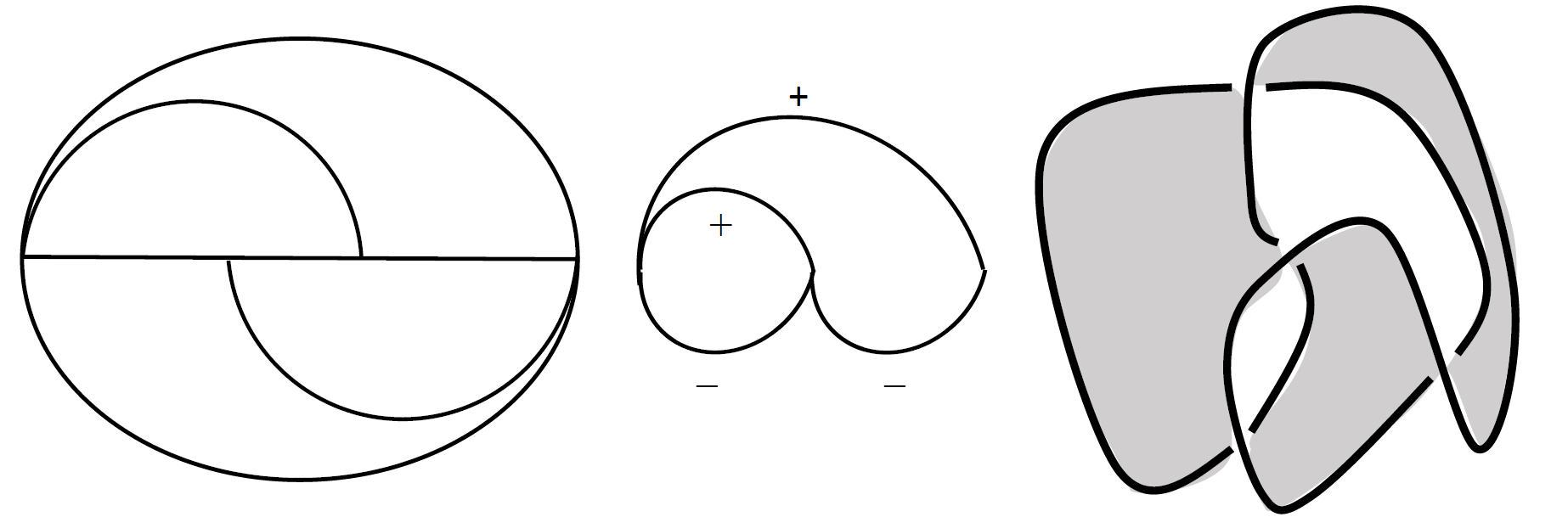}
\caption{The diagram of $x_0$, its Thompson graph and the corresponding shaded link.}
    \label{fig:x0}
\end{figure}

\begin{Lemma}\label{12n+u+1}
Let $L$ be a link which is not an unlink. Assume that $L$ contains $u$ unlinked unknots and has a link diagram with $n$ crossings. Then, there exists a reduced diagram $\Delta$ in $F$ with at most $12n+u+1$ vertices which represents $L$.
\end{Lemma}

\begin{proof}
By Lemma \ref{l:cv} there exists a Thompson graph $\Gamma$ with at most $12n+u$ vertices such that $\La(\Gamma)=L$. Let $\Delta$ be the $(x,x)$-diagram over $\la x\mid x=x^2\ra$ such that $\Gamma(\Delta)=\Gamma$. Clearly, the number of vertices of $\Delta$ is at most $12n+u+1$. As noted above, $\Delta$ does not contain a $\pi\iv\circ\pi$ dipole. Let $\Delta'$ be the reduced diagram equivalent to $\Delta$ and assume that $k$ dipoles (of the form $\pi\circ\pi\iv$) should be canceled to get from $\Delta$ to $\Delta'$. Then, by  Lemma \ref{dipole_link}, the diagram $\Delta'$ represents the link resulting from $L$ by the removal of $k$ unlinked unknots. Note that $\Delta'$ is not the trivial diagram since $L$ is not a family of unlinked unknots.
Let $\Delta''$ be the diagram resulting from $\Delta'$ by the addition of the trivial diagram $\1$ $k$ times on the left. Then, $\Delta''$ is a reduced diagram which by Lemma \ref{plus_one} represents $L$. Since $k$ vertices were erased in the transition from $\Delta$ to $\Delta'$ and $k$ vertices were added in the transition from $\Delta'$ to $\Delta''$, the number of vertices of $\Delta''$ is at most $12n+u+1$.
\end{proof}

\begin{Lemma}
Let $L$ be a link represented by a link diagram with at most $n$ crossings. Assume that $L$ does not contain any unlinked unknot. Then, $L$ is represented by an element $g$ of Thompson group $F$ which satisfies the following.
\begin{enumerate}
\item The normal form of $g$ or its inverse starts with $x_0$.
\item The number of vertices in the reduced diagram of $g$ is at most $12n+1$.
\end{enumerate}
\end{Lemma}

\begin{proof}
Let $\Delta$ be the reduced diagram from Lemma \ref{12n+u+1} which represents $L$. Then, the number of vertices of $\Delta$ is at most $12n+1$. If the normal form of $\Delta$ does not contain $x_0$ nor $x_0\iv$, then, $\Delta=\1\oplus\Delta'$ for some reduced diagram $\Delta'$. By Lemma \ref{plus_one}, the link $L$ represented by $\Delta$ contains an unlinked unknot, a contradiction.
\end{proof}

Let $L$ be a link. Jones \cite{Jo} defined the Thompson index of $L$ as the minimal number of vertices in the diagram representing an element $g$ which represents $L$.  Lemmas \ref{k+2} and \ref{12n+u+1} imply the following.

\begin{Theorem}\label{t:6} The Thompson index of a link containing $u$ unlinked unknots and represented by a link diagram with $n$ crossings does not exceed $12n+u+3$.
\end{Theorem}

\begin{Remark} \label{r:0} It is known \cite{BCS} that the word length of an element $g$ from $F$ in the generators $x_0, x_1$ does not exceed 3 times the number of vertices in the diagram of $g$.  Thus Theorem \ref{t:6} gives a linear upper bound on the word length of an element $g$ representing $L$ in terms of the number of crossings in a link diagram of $L$.
\end{Remark}

\subsection{Some open questions}

\subsubsection{Unlinked elements of $F$}

Let us call an element $g$ of $F$ {\it unlinked} if the link $\call(g)$ is an unlink. The following problem seems difficult

\begin{Problem} \label{q1} Describe the set of unlinked elements of $F$.
\end{Problem}

Theorem \ref{t:6} implies that the problem of recognizing an unlinked element of $F$ is polynomially equivalent to the problem of recognizing an unlink (here elements of $F$ are represented by diagrams or pairs of trees or words in $\{x_0,x_1\}$ and links are represented by link diagrams). By the famous result of Haken \cite{Haken}, the set of unlinked elements is recursive. By \cite{HLP}, it is in NP and by \cite{Kup} it is also in coNP (modulo the generalized Riemann hypothesis).

\subsubsection{Positive links}

One can easily verify that many positive elements of $F$ (i.e., elements whose normal form has trivial negative part) correspond to unlinks (examples: $x_i, x_ix_{i+1}$, etc.). There are, nevertheless, positive elements of $F$ that correspond to non-trivial links (for example, $x_0^2x_2^2$).

\begin{Problem}\label{q2} Which links correspond to positive elements of $F$? When does a positive element of $F$ represent an unlink?
\end{Problem}

\subsubsection{Random links} The Jones' theorem that elements of $F$ represent all links and Remark \ref{r:0} relating the number of crossings in a link diagram with the word length of a corresponding element of $F$ suggests the following

\begin{Problem}\label{q3} Consider a (simple) random walk $w(t)$ on $F$, say, with generators $x_0^{\pm 1}, x_1^{\pm 1}$. What can be said about the link corresponding to $w(t)$? In particular, what is the probability that the link, corresponding to $w(t)$ is an unlink?
\end{Problem}

\begin{minipage}{3 in}
Gili Golan\\
Bar-Ilan University\\
gili.golan@math.biu.ac.il
\end{minipage}
\begin{minipage}{3 in}
Mark Sapir\\
Vanderbilt University\\
m.sapir@vanderbilt.edu
\end{minipage}


\end{document}